\documentclass[letterpaper,12pt]{amsart}
\usepackage[left=1in,right=1in,top=1in,bottom=1in]{geometry}
\usepackage{pgfplots}
\usepackage{tikz, tikz-cd}
\usepackage{qtree}
\usepackage{forest}

\forestset{bintree/.style={for tree={calign=fixed edge angles, grow=south, edge=very thick},
     delay={where content={}{shape=coordinate, for current and siblings={anchor=south}}{}}}}

\forestset{bigbintree/.style={for tree={calign=fixed edge angles, grow=south},
     delay={where content={}{shape=coordinate, for current and siblings={anchor=south}}{}}}}

\newcommand{\bintree}[1]{\vcenter{\hbox{\scalebox{.5}{\tiny\begin{forest}bintree #1 \end{forest}}}}}
\newcommand{\bigbintree}[1]{\vcenter{\hbox{\begin{forest}bigbintree #1 \end{forest}}}}

\newcommand{\tinytree}[1]{
\vcenter{\hbox{
\scalebox{0.4}{
\begin{forest}
  for tree={circle,fill,}
  #1
\end{forest}
}}}}

\usepackage{graphicx}
\usepackage{float}
\usepackage{listings}
\usepackage{multicol}
\usepackage{soul}
\usepackage{pdfpages}

\tikzset{level distance=.7cm,sibling distance=0.5cm}
\usepackage{amsmath, amsfonts, amsthm, amssymb, mathtools, bbold}

\usetikzlibrary{shapes.geometric}

\makeatletter
\newcommand*{\trinum}{}
\DeclareRobustCommand*{\trinum}[1]{
  \ensuremath{
    \mathpalette\@trinum{#1}
  }
}
\newdimen\trinum@sep
\newdimen\trinum@rule
\newcommand*{\@trinum}[2]{
  \setlength{\trinum@rule}{.8\trinum@sep}
  \tikz\node[
    regular polygon,
    regular polygon sides=3,
    draw,
    line width=\trinum@rule,
    inner sep=-0.2mm,
    scale=0.8
  ]{$\m@th#1#2$};
}
\makeatother

\let\oldexists\exists
\let\exists\relax
\let\oldforall\forall
\let\forall\relax

\DeclareMathOperator{\exists}{\oldexists}
\DeclareMathOperator{\forall}{\oldforall}

\usepackage{amsthm}
\newtheorem{definition}{Definition}[section]

\newtheorem{proposition}{Proposition}[section]
\newtheorem{lemma}{Lemma}[section]

\newtheorem{theorem}{Theorem}[section]

\theoremstyle{remark}
\newtheorem{example}{Example}[section]
\newtheorem{remark}{Remark}[section]

\theoremstyle{definition}

\renewcommand{\tilde}{\widetilde}

\renewcommand{\epsilon}{\varepsilon}

\renewcommand{\phi}{\varphi}

\newcommand{\cB}{\mathcal{B}}
\newcommand{\cC}{\mathcal{C}}
\newcommand{\cG}{\mathcal{G}}

\newcommand{\cH}{\mathcal{H}}
\newcommand{\cL}{\mathcal{L}}
\newcommand{\cM}{\mathcal{M}}

\newcommand{\cT}{\mathcal{T}}
\newcommand{\cU}{\mathcal{U}}

\DeclareFontFamily{U}{wncy}{}
\DeclareFontShape{U}{wncy}{m}{n}{<->wncyr10}{}
\DeclareSymbolFont{mcy}{U}{wncy}{m}{n}
\DeclareMathSymbol{\Sha}{\mathord}{mcy}{"58}

\newcommand{\fS}{\mathfrak{S}}

\newcommand{\fN}{\mathfrak{N}}
\newcommand{\fP}{\mathfrak{P}}

\newcommand{\fM}{\mathfrak{M}}
\newcommand{\fF}{\mathfrak{F}}
\newcommand{\one}{\mathbb{1}}

\newcommand{\CoII}[2]{ \Tree [ #1 #2 ] }
\newcommand{\CoIII}[3]{ \Tree [ [ #1 #2 ] #3 ] }

\DeclareMathOperator{\Adm}{Adm}
\DeclareMathOperator{\Orb}{Orb}
\DeclareMathOperator{\Stab}{Stab}
\DeclareMathOperator{\Sym}{Sym}

\newcommand{\ins}{\triangleleft}

\begin{document}

\title{A Hopf algebra on nonplanar binary forests}

\newcommand{\thisdate}{\today}

\author{Elizabeth Xiao}
\address{Department of Mathematics, California Institute of Technology, 1200 E California Blvd, Pasadena, CA 91125}
\email{exiao@caltech.edu}

\begin{abstract}
We equip the graded polynomial algebra generated by nonplanar rooted binary trees with a Hopf algebra structure by defining a coproduct which disallows cutting both children of any given vertex, refining Connes-Kreimer's notion of admissible cuts. We show that the terms in this coproduct have an additional combinatorial interpretation in terms of subsets of leaves, which facilitates the construction of Hopf algebra morphisms involving this Hopf algebra, and creates a connection with a Hopf algebra of Bruned used in the renormalization of stochastic processes. Finally, we show that this Hopf algebra is dual to the universal enveloping algebra of a Lie algebra arising from a pre-Lie operator on binary trees based on edge-insertion.
\end{abstract}

\maketitle

It is generally believed that each type of combinatorial object has a unique natural decomposition associated with it, which preserves some notion of cardinality or grading. We present a coproduct operation which we argue defines a natural decomposition on nonplanar (or abstract) binary forests, which are the fundamental objects in Marcolli, Chomsky and Berwick's mathematical model of Merge, in which the forests represent workspaces of syntactic objects \cite{marcolli2025mathematical}.

Nonplanar binary trees with labelled leaves have also appeared in the literature as phylogenetic trees \cite{billera2001geometry}. They are used to represent a binary operation which is commutative but not associative.

The paper is organized as follows. In Section \ref{prelim}, we introduce the vector space of rooted trees and its symmetric extension to the vector space of forests. These will be the underlying vector spaces for all algebras appearing in this paper. We recall definitions of a Hopf algebra and a pre-Lie algebra, fixing our notation along the way. We keep the Connes-Kreimer Hopf algebra of renormalization in mind as a running example.

In Section \ref{newhopf}, we define a Hopf algebra structure on binary forests in terms of binary-admissible cuts, which disallow cutting both children of any given vertex. In Section \ref{newlie}, we define a pre-Lie operator on binary trees based on insertion into edges. By the Poincar\'{e}-Birkhoff-Witt theorem, the universal enveloping algebra of this pre-Lie algebra is itself a Hopf algebra, which turns out to be dual to the Hopf algebra described in Section \ref{newhopf}. The notion of duality is based on a pairing defined in Section \ref{duality}, where we prove duality by considering the multiplicities of cuts and insertions.

Finally, in Section $\ref{consequences}$ we explore some further consequences of the dual pairing and use it to calculate the coefficients of the pre-Lie exponential of the single-vertex tree.

\section{Preliminaries} \label{prelim}

\subsection{Rooted trees} We assume that graphs are abstract in the sense that we do not fix a planar embedding. A graph is connected if there exists a path between any pair of vertices. A cycle is a path with no repeating edges where the first and last vertices are the same.

A tree is a connected graph with no cycles. A nonplanar rooted tree is an abstract tree graph with a distinguished root vertex. We will draw all trees so that the root is at the top. Choosing a root induces a directed graph structure on a tree; all edges will be oriented to point away from the root.

A forest is a graph whose connected components are trees. Henceforth, we will use the term tree to refer exclusively to nonplanar rooted trees; likewise, a forest will be a forest consisting of nonplanar rooted trees. We can identify forests with multisets of trees.

If $F$ is a forest, we write $V(F)$ for the vertex set of $F$; $E(F)$ for the edge set; $L(F) \subseteq V(F)$ for the subset of leaves, which are vertices of degree one.

Trees may have labels on all or some of its vertices and edges. Most of the trees we examine will only have labels on the leaves; any additional labels will only be used to clarify examples and aid certain proofs.

Given a vertex $v \in V(F)$, the vertex immediately above $v$ is called its \emph{parent}, and the vertices below $v$ are called its \emph{children}. Any vertex below $v$ is called a \emph{descendant} of $v$. If $w \in V(F)$ is a child of $v$, we may refer to the whole \emph{subtree} starting from $w$ as a child of $v$.

Let $\cT$ be some family of rooted nonplanar trees. Let $k$ be a field of characteristic $0$. We write $k\{\cT\}$ for the vector space over $k$ generated by the basis $\cT$. We identify $k[\cT]$ with $\Sym(k\{\cT\})$, the $k$-algebra of polynomials of trees, which is generated as a vector space by forests whose components are in $\cT$.

\subsection{Hopf algebras}
We refer to \cite{CartierPatras} for a thorough description of algebras, coalgebras and Hopf algebras; this section summarizes the main terminology. Let $k$ be a field of characteristic $0$. A \emph{bialgebra} $\cB = (\cB, m, \eta, \Delta, \epsilon)$ over $k$ consists of the following data:
\begin{itemize}
  \item the underlying $k$-vector space $\cB$;
  \item a product $m : \cB \otimes \cB \to \cB$ with unit $\eta : k \to \cB$ (an unitary algebra structure);
  \item a coproduct $\Delta : \cB \to \cB \otimes \cB$ with counit $\epsilon : \cB \to k$ (a counitary coalgebra structure),
\end{itemize}
such that $\Delta$ and $\epsilon$ are morphisms of algebras. We often use \emph{Sweedler notation} to index the coproduct of an element $x \in \cB$ as a sum:
\[ \Delta(x) = \sum_{(x)} x' \otimes x''. \]

Since the unit and counit are assumed to exist, we can write $\cB$ as a triple $\cB = (\cB, m, \Delta)$. We also typically assume that $m$ is associative and $\Delta$ is coassociative. When there is no confusion we generally write $m(x,y)$ as $xy$.

We say $\cB$ is graded as a bialgebra if its underlying vector space is graded; $(\cB, m)$ is graded as an algebra; and $(\cB,\Delta)$ is graded as a coalgebra, in the sense that:
\begin{itemize}
  \item if $x \in \cB_n$ and $y \in \cB_m$, then $m(x,y) \in \cB_{n+m}$;
  \item if $x \in \cB_n$, then $\Delta(x) \in \bigoplus_{k=0}^n \cB_k \otimes \cB_{n-k}$.
\end{itemize}
A graded bialgebra $\cB$ is \emph{connected} if $\cB_0 \cong k$. It is \emph{locally finite} if $\dim \cB_n$ is finite for all $n$.

A \emph{Hopf algebra} $\cH$ is a bialgebra with an antipode map $S : \cH \to \cH$ such that
\[ \eta(\epsilon(x)) =\sum_{(x)} S(x') x'' = \sum_{(x)} x' S(x'') \]
for all $x \in \cH$. If a bialgebra has an antipode, then it is unique. The antipode is an antihomomorphism: for $x, y \in \cH$,
\[ S(xy) = S(y) S(x). \]
If $\cH$ is commutative, then $S$ is a homomorphism.

If $\cB$ is a graded connected bialgebra, then it automatically has an antipode defined recursively by $S(1) = 1$ and
\[ S(x) = -x - \sum_{(x)} S(x') x'' \]
for $x \in \bigoplus_{n=1}^\infty \cB_n$. Thus, any graded connected bialgebra is automatically a graded Hopf algebra.

\subsection{The Connes-Kreimer Hopf algebra} \label{ck}

Let $\cT$ be the set of all unlabelled nonplanar rooted trees. The Connes-Kreimer Hopf algebra $\cH_{CK}$ is defined on the vector space $k[\cT]$. Its product is defined by disjoint unions of forests, which is just the polynomial product on $k[\cT]$. Its coproduct is $\Delta_{CK}$ defined in terms of \emph{admissible cuts} on trees and extended to forests by multiplicativity.

An admissible cut of $T$ is a subset $C \subseteq E(T)$ such that no more than one $e \in C$ appears on a path from the root to any leaf. Consider the forest $F$ produced by deleting all of the edges in $C$ from $T$. This induces a pair $P^C(T) \otimes R^C(T)$, where $R^C(T)$ is the unique tree in $F$ containing the root of $T$, while $P^C(T)$ is a forest of trees lying below the cut edges. The empty cut $C = \varnothing$ is admissible and corresponds to the pair $1 \otimes T$. The set of all admissible cuts in $T$ is called $\Adm(T)$.  In addition to all the admissible cuts, there is also a \emph{total cut} contributing the term $T \otimes 1$.

The Connes-Kreimer coproduct is indexed over all such cuts:
\[ \Delta_{CK}(T) = T \otimes 1 + \sum_{C \in \Adm(T)} P^C(T) \otimes R^C(T) \]

$\cH_{CK}$ is graded by the number of vertices in a forest. It is connected with respect to this grading.

\begin{example}
  Let $T = \tinytree{[[][[]]]}$. Then
  \[
    \Delta_{CK}(T) = T \otimes 1 + 1 \otimes T
     + \bullet \otimes \tinytree{[[[]]]}
     + \bullet \otimes \tinytree{[[][]]}
     + \tinytree{[[]]} \otimes \tinytree{[[]]}
     x+ \bullet \bullet \otimes \tinytree{[[]]} + \bullet \tinytree{[[]]} \otimes \bullet.
  \]
\end{example}

The Connes-Kreimer Hopf algebra was introduced in \cite{connes1999hopf} to describe the structure of renormalization in quantum field theory. Accessible introductions to $\cH_{CK}$ are included in \cite{foissy2013introduction}, \cite{hoffman2003combinatorics}, \cite{yeats2017combinatorial}.

\begin{remark}
  Once we specialize to full binary trees, the vertices will not be explicitly illustrated in our diagrams. However, we will continue to use $\bullet$ to represent the unlabelled single-vertex tree. Labelled single-vertex trees are illustrated by their label.
\end{remark}

\subsection{Variations of Connes-Kreimer for binary trees}

In \cite{marcolli2025mathematical}, Marcolli, Chomsky and Berwick propose three Connes-Kreimer-like coproducts for binary forests. Each of the coproducts is indexed by admissible cuts in the sense of Connes-Kreimer and the pruned forest in the left channel is the same for all three, but each coproduct uses a different notion of quotient in the remainder term in the right channel. They are
\begin{enumerate}
  \item $\Delta^c$, the contraction coproduct, where each subtree in the pruned forest is contracted to a single vertex carrying a trace of the original subtree.
  \item $\Delta^\rho$, the restriction of the Connes-Kreimer coproduct $\Delta_{CK}$ to binary forests.
  \item $\Delta^d$, the deletion coproduct, where the quotient in $\Delta^\rho$ is projected to the unique maximal binary tree obtained by edge contraction.
\end{enumerate}
Note that $\Delta^c$ and $\Delta^d$ are closed under binary trees, while $\Delta^\rho$ is not, since the remainder after cutting is in general not binary.

\begin{example}
 We apply $\Delta^c$ and $\Delta^d$ to $T = \CoIII{$\alpha$}{$\beta$}{$\gamma$}$.
 
 The contraction coproduct is
 \begin{align*}
   \Delta^c(T) =
   T \otimes 1 &+ 1 \otimes T
   + \alpha \otimes \CoIII{\st{$\alpha$}}{$\beta$}{$\gamma$}
   + \beta \otimes \CoIII{$\alpha$}{\st{$\beta$}}{$\gamma$}
   + \gamma \otimes \CoIII{$\alpha$}{$\beta$}{\st{$\gamma$}} 
   + \CoII{$\alpha$}{$\beta$} \otimes \CoII{\st{$\{\alpha,\beta\}$}}{$\gamma$} \\
    &+ \alpha\sqcup\beta \otimes \CoIII{\st{$\alpha$}}{\st{$\beta$}}{$\gamma$}
    + \alpha\sqcup\gamma \otimes \CoIII{\st{$\alpha$}}{$\beta$}{\st{$\gamma$}} 
    + \beta\sqcup\gamma \otimes \CoIII{$\alpha$}{\st{$\beta$}}{\st{$\gamma$}} \\
    &+ \alpha\sqcup\beta\sqcup\gamma \otimes \CoIII{\st{$\alpha$}}{\st{$\beta$}}{\st{$\gamma$}},
 \end{align*}
 while the deletion coproduct produces
  \begin{align*}
   \Delta^d(T) =
   T \otimes 1 &+ 1 \otimes T
   + \alpha \otimes \CoII{$\beta$}{$\gamma$}
   + \beta \otimes \CoII{$\alpha$}{$\gamma$}
   + \gamma \otimes \CoII{$\alpha$}{$\beta$}
   + \CoII{$\alpha$}{$\beta$} \otimes \gamma \\
    &+ \alpha\sqcup\beta \otimes \gamma 
    + \alpha\sqcup\gamma \otimes \beta
    + \beta\sqcup\gamma \otimes \alpha
    + \alpha\sqcup\beta\sqcup\gamma \otimes 1.
 \end{align*}
\end{example}

\subsubsection{Issues with $\Delta^c$}
The contraction coproduct is coassociative: it is in fact the restriction of the Connes-Kreimer Hopf algebra of Feynman graphs to binary trees. This coproduct respects a grading on $\cH_{bin}$ by the number of edges. The odd degree components of the grading are zero, since all full binary trees have an even number of edges.

However, the resulting graded algebra is not connected, since any single-vertex tree has degree $0$. Furthermore, the graded algebra is not even locally finite as a graded vector space, since an arbitrary number of single-vertex trees may be multiplied to an existing tree of degree $2k$ to obtain infinitely many linearly independent elements of degree $2k$. One potential resolution is to quotient this algebra by the ideal generated by $1-\alpha$, where $\alpha$ runs over all single-vertex trees, which produces a connected bialgebra.

Another barrier to local finiteness is the fact that each quotient carries a trace, meaning that infinitely many leaf labels can be generated by taking the coproduct of larger trees. One potential resolution to this is by equipping each tree with a head function and using the values from a head function instead of the trace. Head functions are described in Sections 1.13.3 and 3.4 of \cite{marcolli2025mathematical}.

\subsubsection{Issues with $\Delta^d$}
The deletion coproduct is not coassociative: although $(\Delta^d \otimes \one) \circ \Delta^d$ and $(\one \otimes \Delta^d) \circ \Delta^d$ enumerate the same terms, they may appear with different multiplicity.

For a thorough discussion of these coproducts, including their linguistic interpretations, see Section 1.2 of \cite{marcolli2025mathematical}.

\subsection{Pre-Lie algebras}

Let $\cL$ be a vector space. A bilinear product $\ins : \cL \times \cL \to \cL$ is called a \emph{(right) pre-Lie operator} if it satisfies the \emph{(right) Vinberg identity}
\[
   (x \ins y) \ins z - x \ins (y \ins z) =  (x \ins z) \ins y - x \ins (z \ins y),
\]
making $(\cL,\ins)$ a \emph{(right) pre-Lie algebra}. The expression $(x \ins y) \ins z - x \ins (y \ins z)$ is  the \emph{associator} $A(x,y,z)$ with respect to $\ins$; if $\ins$ is associative, then $A(x,y,z) = 0$ uniformly. The Vinberg identity asserts that the associator is unchanged when the second and third arguments are swapped; that is, for all $x,y,z \in \cL$,
\[ A(x,y,z) = A(x,z,y). \]

A pre-Lie operator induces a Lie bracket defined by
\[ [x,y] := x \ins y - y \ins z, \]
which automatically satisfies antisymmetry and the Jacobi identity, making $\cL$ a Lie algebra. Not all Lie algebras, however, are produced from a pre-Lie algebra. Chapter 6 of \cite{CartierPatras} provides a thorough exposition of pre-Lie algebras.

\subsection{The vertex-insertion pre-Lie algebra}

With $\cT$ as in Section \ref{ck}, let us define a insertion operation on trees. If $T, S \in \cT$ and $v \in V(T)$, define $T \ins_v S$ as the result of creating an edge that joins $v$ to the root of $S$. Then $\ins$ is the sum of all such vertex-insertions
\[ T \ins S = \sum_{v \in V(T)} T \ins_v S. \]
By extending $\ins$ to $k\{\cT\}$, we equip $k\{\cT\}$ with a pre-Lie algebra structure.
\begin{example}
  Let $T = \tinytree{[[][[]]]}$ and $S = \bullet$. Then
  \[ T \ins S = \tinytree{[[][[]][]]} + \tinytree{[[[]][[]]]} + \tinytree{[[][[][]]]} + \tinytree{[[][[[]]]]}. \]
\end{example}
There is an extension of $\ins$ to a noncommutative product $\ast$ on $k[\cT]$, and a cocommutative coproduct $\Sha$ given by unshuffling of monomials; this is one presentation of the universal enveloping algebra $\cU(\cL)$ of the Lie algebra induced by the pre-Lie algebra $(k\{\cT\}, \ins)$. As a consequence of the Poincar\'{e}-Birkhoff-Witt theorem, $\cU(\cL)$ is always a Hopf algebra.  We call $\cH_{GL} = \cU(\cL)= (k[\cT], \ast, \Sha)$ the Grossman-Larson Hopf algebra on trees.

We will not describe the extension of $\ast$ here; it is the result of the general Guin-Oudom construction, which is given in Section \ref{newlie}.

\section{The Hopf algebra of subsets of leaves} \label{newhopf}

We now restrict $\cT$ to the set of binary rooted nonplanar trees, so that $k[\cT]$ is generated as a vector space by binary forests. We will also use $\cT_n$ to refer to the set of binary trees with $n$ leaves. In this section, we will define a Hopf algebra on $k[\cT]$ and refer to it as $\cH = (\cH, \sqcup, \Delta)$. 

\begin{remark}
  If trees in $\cT$ are unlabelled, then there are $|\cT_n| =W_n$ distinct binary trees with $n$ leaves, where $W_n$ is the $n$th Wedderburn-Etherington number (OEIS A001190 \cite{OEIS}). The first values of $W_n$, starting at $W_1$, are
   \[ 1, 1, 1, 2, 3, 6, 11, 23, \dots \]
   corresponding to
   \begin{align*}
     \cT_1 &= \{ \bullet \}, \quad
     \cT_2 = \{ \bintree{[[][]]} \}, \quad
     \cT_3 = \Big\{ \bintree{[[[][]][]]} \Big\}, \quad
     \cT_4 = \Bigg\{ \bintree{[[[[][]][]][]]},  \bintree{[[[][]][[][]]]} \Bigg\}, \\
     \cT_5 &= \Bigg\{ \bintree{[[[[[][]][]][]][]]},  \bintree{[[[[][]][]][[][]]]}, \bintree{[[[[][]][[][]]][]]} \Bigg\}, \dots
   \end{align*}

   $\cH_n$ is generated as a vector space by forests with $n$ leaves. We identify $\cH_0$ with $k$ since it is generated by the unique empty tree (meaning $\cH$ is connected). The values of $\dim \cH_n$, starting from $n=1$, are (OEIS A003214 \cite{OEIS}):
   \[ 1, 1, 2, 3, 6, 10, 20, 37, \dots \]
\end{remark}

The product $\sqcup$ is inherited from the polynomial algebra; the notation indicates that it produces disjoint unions of forests. The forests generating $\cH$ are graded by the number of leaves and the coproduct respects this grading.

The coproduct $\Delta$ on $\cH$ is indexed over all admissible cuts, in the sense of Connes-Kreimer, which do not sever both children of any given vertex, in additional to the unique total cut. The cuts of the first type will be called \emph{binary-admissible cuts}.

\begin{definition}
A \emph{binary-admissible cut} of $T$ is a subset $C \subseteq E(T)$ of edges such that
\begin{enumerate}
\item no pair of edges in $C$ appear on the same path from root to leaf,
\item no two edges in $C$ share the same parent vertex.
\end{enumerate}
Conversely, any subset of $E(T)$ satisfying these two properties defines a binary-admissible cut.
\end{definition}

Let $\Delta_{CK}$ denote the Connes-Kreimer coproduct. Let $T$ be a binary, rooted, non-planar tree: we will imagine that all of its original leaves are labelled (or have some indication that they were originally present), and any new leaves appearing in the second channel of $\Delta_{CK}(T)$ are unlabelled (or have some indication that they are novel, i.e. created from a cut).

\begin{align*}
  \Delta ( \Tree[ [ $\alpha$ $\beta$ ] $\gamma$ ] )
  =
    1 \otimes \Tree[ [ $\alpha$ $\beta$ ] $\gamma$ ]
  &+ \alpha \otimes \Tree [ $\beta$ $\gamma$ ]
  + \beta \otimes \Tree [ $\alpha$ $\gamma$ ]
  + \gamma \otimes \Tree [ $\alpha$ $\beta$ ] \\
  &+ \Tree[ $\alpha$ $\beta$ ] \otimes \gamma
  + \alpha \sqcup \gamma \otimes \beta
  + \beta \sqcup \gamma \otimes \alpha
  + \Tree[ [ $\alpha$ $\beta$ ] $\gamma$ ] \otimes 1
\end{align*}
In order, the terms respectively correspond to the subsets $\varnothing, \{\alpha\}, \{\beta\}, \{\gamma\}, \{\alpha,\beta\},\{\alpha,\gamma\}$ and $\{\alpha,\beta,\gamma\}$.

Let $\Phi$ be a pruning map which maps all trees with novel leaves to $0$, and let $\Pi_{d,\rho}$ be the projection map which sends an at-most binary tree to the unique full binary tree obtained by edge contraction. If the root of a tree with at least two vertices is a novel leaf, then $\Pi_{d,\rho}$ contracts the root, upon which the root label is lost. In particular, when $\Pi_{d,\rho}$ acts on a path with at least one edge, then the result is a single vertex carrying the label of the bottom leaf vertex. This ensures that the correct terms are then deleted by $\Phi$.

Then $\Delta$ is a composition
\[ \Delta = (\one \otimes \Phi) \circ (\one \otimes \Pi_{d,\rho}) \circ \Delta_{CK}. \]

\begin{example}
We will show an example of this sequence of compositions. Let
\[ T = \Tree[ [ $\alpha$ $\beta$ ] $\gamma$ ]. \]
Novel leaves produced by the Connes-Kreimer coproduct are marked by $\ast$.
\begin{align*}
  \Delta^{CK} ( T )
  = 1 \otimes \Tree[ [ $\alpha$ $\beta$ ] $\gamma$ ]
  &+ \Tree[ [ $\alpha$ $\beta$ ] $\gamma$ ] \otimes 1
  + \alpha \otimes \Tree [ .$\ast$ [ $\beta$ $\gamma$ ] ]
  + \beta \otimes \Tree [ [ $\alpha$ ] $\gamma$ ]
  + \gamma \otimes \Tree [ .$\ast$ [ $\alpha$ $\beta$ ] ] \\
  &+ \Tree[ $\alpha$ $\beta$ ] \otimes \Tree [ .$\ast$ $\gamma$ ]
  + \alpha \sqcup \beta \otimes \Tree [ $\ast$ $\gamma$ ]
  + \alpha \sqcup \gamma \otimes \Tree [ .$\ast$ [ .$\cdot$ $\beta$ ] ]
  + \beta \sqcup \gamma \otimes \Tree [ .$\ast$ [ .$\cdot$ $\alpha$ ] ] \\
  &+  \Tree[ $\alpha$ $\beta$ ] \sqcup \gamma \otimes \ast
  + \alpha \sqcup \beta \sqcup \gamma \otimes \Tree [ .$\ast$ $\ast$ ]
\end{align*}
Upon applying $\Pi_{d,\rho}$ to the second channel of the previous expression, we obtain
\begin{align*}
  (\one \otimes \Pi_{d,\rho}) (\Delta^{CK} ( T ))
  = 1 \otimes \Tree[ [ $\alpha$ $\beta$ ] $\gamma$ ]
  &+ \Tree[ [ $\alpha$ $\beta$ ] $\gamma$ ] \otimes 1
  + \alpha \otimes \Tree [ $\beta$ $\gamma$ ]
  + \beta \otimes \Tree [ $\alpha$ $\gamma$ ]
  + \gamma \otimes \Tree [  $\alpha$ $\beta$ ] \\
  &+ \Tree[ $\alpha$ $\beta$ ] \otimes \gamma
  + \alpha \sqcup \beta \otimes \Tree [ $\ast$ $\gamma$ ]
  + \alpha \sqcup \gamma \otimes \beta
  + \beta \sqcup \gamma \otimes \alpha \\
  &+  \Tree[ $\alpha$ $\beta$ ] \sqcup \gamma \otimes \ast
  + \alpha \sqcup \beta \sqcup \gamma \otimes \ast.
\end{align*}
Note, in particular, how $\Pi_{d,\rho}$ acts on the last two entries. Finally, $\Phi$ deletes all of the trees which still have novel leaves marked $\ast$ in the second channel, so
\begin{align*}
  (\one \otimes \Phi ) ((\one \otimes \Pi_{d,\rho}) (\Delta^{CK} ( T )))
  &= 1 \otimes \Tree[ [ $\alpha$ $\beta$ ] $\gamma$ ]
  + \Tree[ [ $\alpha$ $\beta$ ] $\gamma$ ] \otimes 1
  + \alpha \otimes \Tree [ $\beta$ $\gamma$ ]
  + \beta \otimes \Tree [ $\alpha$ $\gamma$ ] \\
  &\quad + \gamma \otimes \Tree [  $\alpha$ $\beta$ ]
  + \Tree[ $\alpha$ $\beta$ ] \otimes \gamma
  + \alpha \sqcup \gamma \otimes \beta
  + \beta \sqcup \gamma \otimes \alpha \\
  &= \Delta(T).
\end{align*}
\end{example}

\begin{proposition}
  The terms of $\Delta(T)$ are in bijection with subsets of leaves of $T$.
\end{proposition}

\begin{proof}
When we speak of $L(T)$, the set of leaves of $T$, we can implicitly assume that the leaves have distinct labels and identify each leaf with its label. Strictly speaking, if there are any repeated labels among the leaves of $T$, we assume that the leaves are distinct objects even if they carry the same label.

For a tree with a single leaf, the only cuts are the empty and total cuts, which correspond to the empty set and the singleton subset with the unique leaf.

Before proceeding with the inductive step, it is helpful to illustrate it with the example of a cherry tree $T = \Tree [ $\alpha$ $\beta$ ]$. The only non-trivial cuts of $T$ involve cutting either $\alpha$ or $\beta$. Both leaves cannot be simultaneously cut, as that would entail severing both children of the root. Hence
\[ \Delta(T) = 1 \otimes T + T \otimes 1 + \alpha \otimes \beta + \beta \otimes \alpha, \]
where the terms respectively correspond to $\varnothing, \{\alpha,\beta\}, \{\alpha\}$ and $\{\beta\}$.

Now let $T$ be a tree with $n\geq 3$ leaves and suppose that the result holds for all trees with fewer than $n$ leaves. Let $T_\ell$ and $T_r$ be the left and right subtrees of $T$ respectively. Since $T_\ell$ and $T_r$ have strictly fewer leaves than $G$, the inductive hypothesis applies to them.

 Consider a cut of $T$. Since we are not allowed to cut both children of the root, the total cut corresponds uniquely to $L(T) = L(T_\ell) \sqcup L(T_r)$. Any other cut of $T$ corresponds to a unique cut of $T_\ell$ and $T_r$ (either of which may be empty), and conversely. Likewise, any subset of $L(T)$ is the disjoint union of a subset of $L(T_\ell)$ with a subset of $L(T_r)$.
\end{proof}

\begin{theorem}
  $\Delta$ is coassociative.
\end{theorem}

\begin{proof}
  Consider two consecutive cuts along the same path from the root to a leaf. If the two cuts occur at what were non-adjacent edges on the original tree, then no matter which edge is cut first, the other edge remains unmodified and remains available for the second cut. Since the order of cutting is irrelevant, the same term appears in both $(\Delta \otimes \one) \circ \Delta$ and $(\one \otimes \Delta) \circ \Delta$.

  Now consider a term in $(\Delta \otimes \one) \circ \Delta$ consisting of two cuts, where the second cut occurs immediately below the root of the subtree obtained from the first cut. We claim that the same term appears in $(\one \otimes \Delta) \circ \Delta$. In the illustration, the two edges to be cut are labelled $1$ and $2$ in the original tree.
  
  Suppose the bottom edge is cut first, as in the left-hand picture: its sibling and its parent become the same edge in the quotient tree. Cutting this combined edge leaves the same quotient as if the top edge had been cut in the first place.

\tikzset{Triangle/.style = {draw, regular polygon, regular polygon sides = 3, inner sep = 1pt, anchor=north}}

\begin{figure}
\begin{equation*}
\begin{tikzpicture}
  [level distance=10mm,
   sibling distance=22mm,
   edge from parent/.style = {
    draw, edge from parent path = {(\tikzparentnode) -- (\tikzchildnode.north)}},
   ]
  \coordinate
   child {
     child {
       child {
         child {
           node[Triangle] {$A$}
           edge from parent node[auto=right,pos=.6] {$1$}
         }
         child {
           node[Triangle] {$B$}
           edge from parent node[auto=left,pos=.6] {$4$}
         }
         edge from parent node[auto=right,pos=.6] {$2$}
       }
       child {
         node[Triangle] {$C$}
         edge from parent node[auto=left,pos=.6] {$5$}
       }
       edge from parent node[auto=right,pos=.6] {$3$}
     }
     child {
       node {$...$}
     }
   edge from parent node[fill=white] {\reflectbox{$...$}}
   }
   child {node {$...$}};
\end{tikzpicture}
\end{equation*}

\begin{minipage}[b]{.47\textwidth}
  \centering
  
\begin{tabular}{@{}r@{}}
cut at $1$ \\
combine $2, 4$
\end{tabular}
\Bigg\downarrow
$\Delta$
\vspace{6mm}
  
\begin{math}
\vcenter{\hbox{
\begin{tikzpicture}
  \node[Triangle] {$A$};
\end{tikzpicture}
}}
\otimes
\vcenter{\hbox{
\begin{tikzpicture}
  [level distance=10mm,
   sibling distance=22mm,
   edge from parent/.style = {
    draw, edge from parent path = {(\tikzparentnode) -- (\tikzchildnode.north)}},
   ]
  \coordinate
   child {
     child {
       child {
           node[Triangle] {$B$}
           edge from parent node[auto=right,pos=.6] {$2\sim4$}
       }
       child {
         node[Triangle] {$C$}
         edge from parent node[auto=left,pos=.6] {$5$}
       }
       edge from parent node[auto=right,pos=.6] {$3$}
     }
     child {
       node {$...$}
     }
   edge from parent node[fill=white] {\reflectbox{$...$}}
   }
   child {node {$...$}};
\end{tikzpicture}
}}
\end{math}

\vspace{6mm}
cut at $2\sim4$
\Bigg\downarrow
$\one \otimes \Delta$
\vspace{6mm}

\resizebox{\textwidth}{!}{
\begin{math}
\vcenter{\hbox{
\begin{tikzpicture}
  \node[Triangle] {$A$};
\end{tikzpicture}
}}
\otimes
\vcenter{\hbox{
\begin{tikzpicture}
  \node[Triangle] {$A$};
\end{tikzpicture}
}}
\otimes
\vcenter{\hbox{
\begin{tikzpicture}
  [level distance=10mm,
   sibling distance=22mm,
   edge from parent/.style = {
    draw, edge from parent path = {(\tikzparentnode) -- (\tikzchildnode.north)}},
   ]
  \coordinate
   child {
     child {
       node[Triangle] {$C$}
       edge from parent node[auto=right,pos=.6] {$3\sim5$}
     }
     child {
       node {$...$}
     }
   edge from parent node[fill=white] {\reflectbox{$...$}}
   }
   child {node {$...$}};
\end{tikzpicture}
}}
\end{math}
}
\end{minipage}
\begin{minipage}[b]{.02\textwidth}
  $=$
\end{minipage}
\begin{minipage}[b]{.47\textwidth}
    \centering
    
$\Delta$
\Bigg\downarrow
cut at $2$
\vspace{10mm}
      
\resizebox{\textwidth}{!}{
\begin{math}
\vcenter{\hbox{
\begin{tikzpicture}
  [level distance=10mm,
   sibling distance=22mm,
   edge from parent/.style = {
    draw, edge from parent path = {(\tikzparentnode) -- (\tikzchildnode.north)}},
   ]
  \coordinate
         child {
           node[Triangle] {$A$}
           edge from parent node[auto=right,pos=.6] {$1$}
       }
       child {
         node[Triangle] {$B$}
         edge from parent node[auto=left,pos=.6] {$4$}
       };
\end{tikzpicture}
}}
\otimes
\vcenter{\hbox{
\begin{tikzpicture}
  [level distance=10mm,
   sibling distance=22mm,
   edge from parent/.style = {
    draw, edge from parent path = {(\tikzparentnode) -- (\tikzchildnode.north)}},
   ]
  \coordinate
   child {
     child {
       node[Triangle] {$C$}
       edge from parent node[auto=right,pos=.6] {$3\sim5$}
     }
     child {
       node {$...$}
     }
   edge from parent node[fill=white] {\reflectbox{$...$}}
   }
   child {node {$...$}};
\end{tikzpicture}
}}
\end{math}
}

\vspace{10mm}
$\Delta \otimes \one$
\Bigg\downarrow
\begin{tabular}{@{}l@{}}
cut at $4$ \\
contract $1$
\end{tabular}
\vspace{10mm}

\resizebox{\textwidth}{!}{
\begin{math}
\vcenter{\hbox{
\begin{tikzpicture}
  \node[Triangle] {$A$};
\end{tikzpicture}
}}
\otimes
\vcenter{\hbox{
\begin{tikzpicture}
  \node[Triangle] {$A$};
\end{tikzpicture}
}}
\otimes
\vcenter{\hbox{
\begin{tikzpicture}
  [level distance=10mm,
   sibling distance=22mm,
   edge from parent/.style = {
    draw, edge from parent path = {(\tikzparentnode) -- (\tikzchildnode.north)}},
   ]
  \coordinate
   child {
     child {
       node[Triangle] {$C$}
       edge from parent node[auto=right,pos=.6] {$3\sim5$}
     }
     child {
       node {$...$}
     }
   edge from parent node[fill=white] {\reflectbox{$...$}}
   }
   child {node {$...$}};
\end{tikzpicture}
}}
\end{math}
}     
 \end{minipage}
 \end{figure}

  Our next claim is that $(\Delta \otimes \one) \circ \Delta$ has at least as many terms as $(\one \otimes \Delta) \circ \Delta$. Since we have matched each term in $(\Delta \otimes \one) \circ \Delta$ to a unique term in $(\one \otimes \Delta) \circ \Delta$, this is sufficient to show that they are equal.

  Recall that $\Delta$ can be written as a composition in terms of the Connes-Kreimer coproduct, Let $\tilde\Phi := \Phi \circ \Pi_{d,\rho}$. Then
  \begin{align*}
    \Delta &= (\one \otimes \Phi) \circ (\one \otimes \Pi_{d,\rho}) \circ \Delta_{CK} \\
    &= (\one \otimes \Phi \circ \Pi_{d,\rho} ) \circ \Delta_{CK} \\
    &= (\one \otimes \tilde\Phi) \circ \Delta_{CK}
  \end{align*}
  We also recall that the Connes-Kreimer coproduct is coassociative. Then
  \begin{align*}
    (\Delta \otimes \one) \circ \Delta
    &= (\Delta \otimes \one) \circ (\one \otimes \tilde\Phi) \circ \Delta_{CK} \\
    &= ((\one \otimes \tilde\Phi) \circ \Delta _{CK} \otimes \one) \circ (\one \otimes \tilde\Phi) \circ \Delta_{CK} \\
    &= ((\one \otimes \tilde\Phi) \circ \Delta _{CK} \otimes \tilde\Phi) \circ \Delta_{CK} \\
    &= (\one \otimes \tilde\Phi \otimes \tilde\Phi) \circ (\Delta_{CK} \otimes \one)\circ \Delta_{CK} \\
    &= (\one \otimes \tilde\Phi \otimes \tilde\Phi) \circ (\one \otimes \Delta_{CK})\circ \Delta_{CK},
  \end{align*}
  where the last equality invokes the coassociativity of Connes-Kreimer. Meanwhile,
  \begin{align*}
    (\one \otimes \Delta)\Delta
    &= (\one \otimes ((\one \otimes \tilde\Phi) \circ \Delta_{CK})) \circ \Delta \\
    &= (\one \otimes ((\one \otimes \tilde\Phi) \circ \Delta_{CK})) \circ (\one \otimes \tilde\Phi) \circ \Delta_{CK} \\
    &= (\one \otimes \one \otimes \tilde\Phi ) \circ (\one \otimes \Delta_{CK}) \circ (\one \otimes \tilde\Phi) \circ \Delta_{CK} \\
    &= (\one \otimes \tilde\Phi \otimes \tilde\Phi ) \circ (\one \otimes \Delta_{CK}) \circ (\one \otimes \tilde\Phi) \circ \Delta_{CK}.
  \end{align*}
  The last equality is justified because any term which appears in $(\one \otimes \Delta) \circ \Delta$ has no component containing any trees with novel leaves, nor any trees which are not binary, so applying $\tilde\Phi$ to the middle component does not change any term.
  
  Since $\tilde\Phi$ serves to contract certain edges and to delete certain terms, $(\one \otimes \tilde\Phi) \circ \Delta_{CK}$ cannot have more terms than $\Delta_{CK}$. A further application of $\Delta_{CK}$ in the second channel also cannot introduce more terms, nor can $\one \otimes \tilde\Phi \otimes \tilde\Phi$ preserve more terms than above.
  
  Hence $(\Delta \otimes \one) \circ \Delta$ has at least as many terms as $(\one \otimes \Delta) \circ \Delta$, and they are equal.
  
  The same principle applies to pairs of simultaneous cuts that result in a remainder consisting of forests instead of a singular tree, so $\Delta$ is coassociative.
\end{proof}

\begin{remark}
  Joachim Kock points out that binary forests are a \emph{restriction species} in the sense of Schmitt \cite{schmitt1993hopf}, where the restriction of a subset corresponds to extracting certain subtrees and contracting the remaining edges. This gives an alternative proof of coassociativity.
\end{remark}

\subsection{The iterated coproduct}

The $n$th iterated coproduct is the map $\Delta^{(n)} : \cH \to \cH^{\otimes (n+1)}$ defined recursively by $\Delta^{(1)} := \Delta$ and
\[ \Delta^{(k)} := (\Delta^{(k-1)} \otimes \one) \circ \Delta \]
for $k \geq 2$. We can just as easily define $\Delta^{(k)}$ by $(\one \otimes \Delta^{(k-1)}) \circ \Delta$; coassociativity ensures that this produces the same sequence of maps.

Given a finite set $S$, a \emph{partition} of $S$ is an unordered collection of pairwise disjoint nonempty sets $S_1, \dots, S_m$ whose union is $S$. A \emph{composition} of $S$ is an ordered partition of $S$. The iterated coproduct of $T$ is indexed by set compositions of $L(T)$.

\subsection{The antipode} Since $\cH$ is a graded connected Hopf algebra, the antipode can be defined recursively by $S(1) = 1$ and
\[ S(T) = -T - \sum_{(T)} S(T') T'' \] 
where the sum is indexed over all terms in the reduced coproduct of $T$, i.e. all terms other than the empty and full cuts.

\begin{example}
Here are the antipodes of some small trees with up to three leaves:
\begin{align*}
  S(\alpha) &= -\alpha \\
  S( \Tree [  $\alpha$ $\beta$ ] )
  &= - \Tree [ $\alpha$ $\beta$ ] + 2 \alpha \sqcup \beta \\
  S( \Tree [ [ $\alpha$ $\beta$ ] $\gamma$ ] )
  &= - \Tree [ [ $\alpha$ $\beta$ ] $\gamma$ ]
    + \alpha \sqcup \Tree [ $\beta$ $\gamma$ ]
    + \beta \sqcup \Tree [ $\alpha$ $\gamma$ ]
    + 2\gamma \sqcup \Tree [ $\alpha$ $\beta$ ]
    - 4 \alpha \sqcup \beta \sqcup \gamma.
\end{align*}
\end{example}

The terms in $S(T)$ correspond bijectively to the total cuts of $T$ that do not sever both children of any node. We will call such cuts \emph{binary-total cuts}.

Each tree in the forest resulting from a binary-total cut is then projected by $\Pi_{d,\rho}$ to the unique full binary tree obtained by edge contraction. (Compare this to the behaviour of the Connes-Kreimer antipode, which is indexed by all total cuts of a given tree.)

The empty cut is considered to be the unique binary-total cut where no edges are severed, and it corresponds to the unique term of $T$ in $S(T)$, which always carries a coefficient of $-1$.

We will formalize this observation, as well as some others, in the following proof:

\begin{theorem}
  If $T$ is a tree with $n$ leaves, then:
  \begin{enumerate}
    \item Each term of $S(T)$ corresponds to a unique binary-total cut of $T$, and vice versa.
    \item The sign of the coefficient in a term in $S(T)$ is negative if the term is a forest consisting of an odd number of trees, and positive if the forest has an even number of trees.
    \item The sum of the absolute values of the coefficients in $S(T)$ is $3^{n-1}$.
  \end{enumerate}
\end{theorem}

\begin{proof}
  These properties are easily seen for $n=1$, when the tree is a single leaf and the only binary-total cut is the empty cut. For $n=2$, consider $S(\Tree [  $\alpha$ $\beta$ ])$: there is the empty cut and two copies of $\alpha \sqcup \beta$, one resulting from cutting the edge above $\alpha$, and the other from cutting the edge above $\beta$. We proceed with a general proof for larger trees.

  \begin{enumerate}
    \item Let $T$ be a tree with $n \geq 2$ leaves. We will establish a bijection between non-empty total cuts of $T$ and pairs $(T'', C_{T'})$, where $T'' \otimes T'$ is a term of $\tilde \Delta(T)$ and $C_{T'}$ is a total cut of the forest $T'$. In simpler terms, we show that each binary-total cut of $T$ has a unique representation as a binary-admissible cut of $T$ along with a binary-total cut of the severed forest.
    
    Since $T''$ is a forest of trees with at most $n-1$ leaves each and $S$ is multiplicative, each term of $S(T')$ corresponds to a binary-total cut of the forest $T''$. By the recursive definition of $S(T)$, each term apart from $-T$ corresponds to a pair $(T'',C_{T'})$.

    Let $C$ be a total cut of $T$. Let $C' \subseteq C$ be the set of all edges which are the first edge on any path from the root. Then $C'$ describes an admissible cut of $T$. Let $T' \otimes T''$ be the term in $\tilde\Delta(T)$ corresponding to the cut $C$; then $C_{T'} := C \setminus C'$ is a binary-total cut of $T'$, giving a pair $(T'', C_{T'})$.

    Conversely, given a pair $(T'', C_{T'})$, let $C'$ be the admissible cut which produces $T' \otimes T''$. Since all edges in $C'$ appear above $T''$, $C' \cup C_{T'}$ is a binary-total cut of $T$. Moreover, since each total cut can be built up recursively as a sequence of admissible cuts over the severed forests, the terms in $S(T)$ are exactly the results of performing a simultaneous binary-total cut on $T$, and then projecting each tree to its image under $\Pi_{d,\rho}$.
    
    \item If a term in $S(T')$ is a forest of $k$ trees, then it produces a term in $-S(T') T''$ with $k+1$ trees and the opposite coefficient.
    
    Since $T$ appears in $S(T)$ with a negative coefficient, all odd forests have a negative coefficient, while even forests carry a positive coefficient.

    \item For the sum of the unsigned coefficients: each binary-total cut is uniquely determined by a choice of three options at each internal node, and conversely each set of decisions determines a total cut. At each internal node, of which there are $n-1$ for a tree with $n$ leaves, we can choose either to keep both of its children intact, or to cut exactly one of its children. Hence there are $3^{n-1}$ possible binary-total cuts.
  \end{enumerate}
\end{proof}

\begin{example}
We enumerate all of the binary-total cuts on
 $T = \CoIII{$\alpha$}{$\beta$}{$\gamma$}$
and match them with the corresponding terms in $S(T)$.

    \begin{center}
    \begin{tabular}{ |c||c| }
       \hline
       binary-total cut & term in $S(T)$ \\ \hline\hline
       \Tree [ [ $\alpha$ $\beta$ ] $\gamma$ ]
         & \Tree [ [ $\alpha$ $\beta$ ] $\gamma$ ] $\quad$ (empty cut) \\ \hline
       \Tree [ [ $\alpha$ $\beta$ ] $\gamma$ ]
       \makebox[0pt][r]{\raisebox{0.4em}{\color{red}\tiny$\diagdown$}\hspace*{1.6em}}
         & $\gamma \sqcup \CoII{$\alpha$}{$\beta$}$ \\ \hline
       \Tree [ [ $\alpha$ $\beta$ ] $\gamma$ ]
       \makebox[0pt][r]{\raisebox{0.4em}{\color{red}\tiny$\diagup$}\hspace*{0.6em}}
         & $\gamma \sqcup \CoII{$\alpha$}{$\beta$}$ \\ \hline
       \Tree [ [ $\alpha$ $\beta$ ] $\gamma$ ]
       \makebox[0pt][r]{\raisebox{0em}{\color{red}\tiny$\diagdown$}\hspace*{2.6em}}
         & $\alpha \sqcup \CoII{$\beta$}{$\gamma$}$ \\ \hline
       \Tree [ [ $\alpha$ $\beta$ ] $\gamma$ ]
       \makebox[0pt][r]{\raisebox{0em}{\color{red}\tiny$\diagup$}\hspace*{1.6em}}
         & $\beta \sqcup \CoII{$\alpha$}{$\gamma$}$ \\ \hline
       \Tree [ [ $\alpha$ $\beta$ ] $\gamma$ ]
       \makebox[0pt][r]{\raisebox{0.4em}{\color{red}\tiny$\diagdown$}\hspace*{1.6em}}\makebox[0pt][r]{\raisebox{0em}{\color{red}\tiny$\diagdown$}\hspace*{2.6em}}
         & $\alpha\sqcup\beta\sqcup\gamma$ \\ \hline
       \Tree [ [ $\alpha$ $\beta$ ] $\gamma$ ]
       \makebox[0pt][r]{\raisebox{0.4em}{\color{red}\tiny$\diagdown$}\hspace*{1.6em}}\makebox[0pt][r]{\raisebox{-0.2em}{\color{red}\tiny$\diagup$}\hspace*{1.8em}}
         & $\alpha\sqcup\beta\sqcup\gamma$ \\ \hline
       \Tree [ [ $\alpha$ $\beta$ ] $\gamma$ ]
       \makebox[0pt][r]{\raisebox{0.4em}{\color{red}\tiny$\diagup$}\hspace*{0.6em}}\makebox[0pt][r]{\raisebox{0em}{\color{red}\tiny$\diagdown$}\hspace*{2.6em}}
         & $\alpha\sqcup\beta\sqcup\gamma$ \\ \hline
        \Tree [ [ $\alpha$ $\beta$ ] $\gamma$ ]
       \makebox[0pt][r]{\raisebox{0.4em}{\color{red}\tiny$\diagup$}\hspace*{0.6em}}\makebox[0pt][r]{\raisebox{0em}{\color{red}\tiny$\diagup$}\hspace*{1.6em}}
         & $\alpha\sqcup\beta\sqcup\gamma$ \\ \hline
    \end{tabular}
    \end{center}
\end{example}

\subsection{Subalgebra of comb trees}

The comb trees are the binary trees which can be represented as having all of its leaves, except one of the two unique leaves at the lowest depth, leaning towards the right.

We denote the unique abstract unlabelled comb tree with $n$ leaves by $C_n$. The first four unlabelled comb trees are
\[
   C_1 = \bullet, \quad
   C_2 = \bintree{[[][]]}, \quad
   C_3 = \bintree{[[[][]][]]}, \quad
   C_4 = \bintree{[[[[][]][]][]]}.
\]

We also make the identification $C_0 = 1$. The comb trees are a Hopf subalgebra of $\cH$, since any cut of $C_n$ produces a term of the form $C_m C_1^i \otimes C_{n-m-i}$ in the coproduct (where it is possible that $m=0$).

An explicit formula for the coproduct of $C_n$ is
\begin{align*}
  \Delta(C_n) = \sum_{i=0}^{n-2} {n - 2 \choose i} C_1^i \otimes C_{n-i} &+ 2 \sum_{j=0}^{n-2} {n-2 \choose j} C_1^{j+1} \otimes C_{n-j-1} \\
  &+ \sum_{k=2}^{n-1} \left( \sum_{\ell=0}^{n-k-1} {n-k-1 \choose \ell} C_k C_1^\ell \otimes C_{n-k-\ell} \right) + C_n \otimes 1
\end{align*}

Pascal's recursion formula allows us to partially combine the first two sums, resulting in the equivalent expression
\begin{align*}
  \Delta(C_n) = \sum_{i=1}^{n-2} \left[ {n - 2 \choose i-1} + {n-1 \choose i} \right]  C_1^i \otimes C_{n-i} &+ \sum_{k=2}^{n-1} \left( \sum_{\ell=0}^{n-k-1} {n-k-1 \choose \ell} C_k C_1^\ell \otimes C_{n-k-\ell} \right) \\
  &+ 2 C_1^{n-1} \otimes C_1 +  1 \otimes C_n +  C_n \otimes 1.
\end{align*}

\subsection{A first polynomial invariant}

The polynomial algebra $k[X]$ is equipped with a Hopf algebra structure given the coproduct $\Delta(X) = 1 \otimes X + X \otimes 1$. Since the terms in the coproduct are indexed by subsets of the leaves, the map sending any forest with $n$ leaves to $X^n$ defines a surjective morphism of Hopf algebras $\cH \to k[X]$.

We also observe that in the unlabelled version of $\cH$, the Hopf subalgebra generated by single-vertex trees is an injective copy of $k[X]$ inside $\cH$.

Suppose that $\cH$ has $m$ possible leaf labels $\ell_1, \dots, \ell_m$. Then there is a surjective morphism $\cH \to k[X_1, \dots, X_m]$. Similarly there is an injective copy of $k[X_1, \dots, X_m]$ in $\cH$, where $X_i$ maps to the single vertex tree with label $\ell_i$.

\subsection{Connection with Bruned's Connes-Kreimer-like Hopf algebra}

Bruned defines in his PhD thesis a coproduct $\delta^+$ on rooted nonplanar trees whose terms are also indexed by subsets of leaves \cite{bruned2015singular}. This coproduct is applied to the renormalization of stochastic processes.

We base our exposition on the notation in \cite{foissy2021algebraic}, which provides another description of Bruned's construction. Note that both \cite{bruned2015singular} and \cite{foissy2021algebraic} use the convention opposite from ours where the remainder term containing the root in the coproduct occurs in the left channel.

For a tree in $\cH_{CK}$, define $\delta^+$ as
\[ \delta^+(T) = T \otimes 1 + \sum_{C \in \mathrm{Adm}'(T)} B^+(P^C(T)) \otimes R^C(T), \]
where $\mathrm{Adm}'(T)$ denotes the subset of admissible cuts, in the sense of Connes-Kreimer, of $T$ where the non-root leaves of $R^c(T)$ are leaves of $T$. These are exactly the admissible cuts that do not sever all of the children of a given node, since severing all of the children of a node turns it into a new leaf. The term $P^C(T) \otimes R^C(T)$ is the term corresponding to the admissible cut $C$ in $\Delta_{CK}$; the operator $B^+$ is the $1$-cocycle which joins all trees in a forest to a common root.

\begin{example}
 Let $T = \tinytree{[[[,label={[font=\Huge]south:$1$}][,label={[font=\Huge]south:$2$}]][[,label={[font=\Huge]south:$3$}]]]}$. Then
 \begin{align*}
    \delta^+(T) = T \otimes 1 + 1 \otimes T
    &+ \bullet_1 \otimes \tinytree{[[[,label={[font=\Huge]south:$2$}]][[,label={[font=\Huge]south:$3$}]]]}
    + \bullet_2 \otimes \tinytree{[[[,label={[font=\Huge]south:$1$}]][[,label={[font=\Huge]south:$3$}]]]}
    + \tinytree{[[,label={[font=\Huge]south:$3$}]]} \otimes 
    \tinytree{[[[,label={[font=\Huge]south:$1$}][,label={[font=\Huge]south:$2$}]]]} \\
    &+ \tinytree{[[,label={[font=\Huge]south:$1$}][,label={[font=\Huge]south:$2$}]]} \otimes \tinytree{[[[,label={[font=\Huge]south:$3$}]]]} 
    + \tinytree{[[,edge=dashed,label={[font=\Huge]south:$1$}][,edge=dashed[,label={[font=\Huge]south:$3$}]]]} \otimes \tinytree{[[[,label={[font=\Huge]south:$2$}]]]} 
    + \tinytree{[[,edge=dashed,label={[font=\Huge]south:$2$}][,edge=dashed[,label={[font=\Huge]south:$3$}]]]} \otimes \tinytree{[[[,label={[font=\Huge]south:$1$}]]]} .
  \end{align*}
  The dashed edges are the ones created by $B^+$.
\end{example}

Then our coproduct $\Delta$ is the composition $(B^- \otimes \Pi_{d,\rho}) \circ \delta^+$, where $B^-$ is the operator that deletes all roots from a forest and $\delta^+$ is restricted to binary trees in $\cH_{CK}$. 

Consider a quotient of $\cH_{CK}$ where all directed paths in a graph, where no originating vertex has more than one tree, are contracted into a single edge. For instance, all ladder graphs are identified with a single vertex. Specifically, this quotient $\cH_{CK}/ \sim$ is generated as an algebra by \emph{reduced trees}, which are trees with no vertices of degree $2$ other than possibly the root.

\begin{example}
  The following trees in $\cH_{CK}$ are all in the same class under the quotient $\sim$:
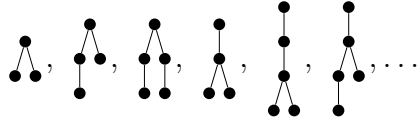
  \[ \tinytree{[[][]]}, \tinytree{[[[]][]]}, \tinytree{[[[]][[]]]}, \tinytree{[[[][]]]}, \tinytree{[[[[][]]]]}, \tinytree{[[[[]][]]]}, \dots \]
\end{example}

Any edge of a tree in $\cH_{CK}$ which has a single child is not eligible to be cut under $\delta^+$, so this quotient respects $\delta^+$. 

\begin{remark}
  While it is tempting to extend our notion of binary-admissible cuts to trees of arbitrary arity by disallowing cuts that sever all of the children of any given node, this has unwelcome consequences for duality, since the edge insertion procedure of Section \ref{newlie} can only create new vertices with binary branching. Thus we limit our consideration to binary trees.
\end{remark}

\section{The edge-insertion Lie algebra} \label{newlie}

In the previous section, we constructed a new Hopf algebra $\cH$ by explicitly defining its coproduct and verifying its coassociativity. Later we will present an alternative construction of $\cH$ as the dual of the universal enveloping algebra of a Lie algebra. This section deals with the construction of the Lie algebra on binary trees; duality is proved in Section \ref{duality}.

We will refer to this pre-Lie algebra on $k\{\cT\}$ as $(\cL, \ins)$.

\subsection{The edge-insertion operation}

Let $T, S \in \cT$ and let $e$ be an edge of $T$. To insert $S$ into $T$ at $e$, an operation which we denote by $T \ins_e S$, split the edge $e$ to create a new vertex and join the root of $S$ to this vertex by a new edge. The following is an example of this:

\begin{align*}
\begin{tikzpicture}
  [baseline=-5mm,
   level distance=3.5mm,
   sibling distance=13mm,
   every node/.style={text height=0.5mm}]
  \coordinate
       child {
         child {node {$\alpha$}}
         child {node {$\beta$}}
         edge from parent node[left] {$e$}
       }
       child {node {$\gamma$}}
       ;
\end{tikzpicture}
  \ins_e
  \Tree [ $\delta$ $\epsilon$ ]
  &=
  \Tree [ [ [ $\alpha$ $\beta$ ] [ $\delta$ $\epsilon$ ] ] $\gamma$ ]
\end{align*}

Define $T \ins S$ as the sum of the trees produced by inserting $S$ at each edge of $T$, \emph{plus} the tree obtained by grafting $T$ and $S$ to a common root (which coincides with $\fM(T,S)$, the result of applying Merge to $T$ and $S$). For instance,

\begin{align*}
  \Tree [ [ $\alpha$ $\beta$ ] $\gamma$ ]
  \ins
  \Tree [ $\delta$ $\epsilon$ ]
  =
  \Tree [ [ [ $\alpha$ $\beta$ ] [ $\delta$ $\epsilon$ ] ] $\gamma$ ]
  + \Tree [ [ $\alpha$ $\beta$ ] [ [ $\delta$ $\epsilon$ ] $\gamma$ ] ]
  + \Tree [ [ [ $\alpha$ [ $\delta$ $\epsilon$ ] ] $\beta$ ] $\gamma$ ]
  &+ \Tree [ [ $\alpha$ [ [ $\delta$ $\epsilon$ ] $\beta$ ] ] $\gamma$ ] \\
  &
  + 
  \Tree [ [ [ $\alpha$ $\beta$ ] $\gamma$ ] [ $\delta$ $\epsilon$ ] ]
\end{align*}
Indeed, the last term is equal to $\fM \Big(  \Tree [ [ $\alpha$ $\beta$ ] $\gamma$ ], \Tree [ $\delta$ $\epsilon$ ] \Big)$.

Every instance of insertion turns one edge into two and then creates a new one, meaning the total edge count increases by two. Again, there is a natural grading on trees by the number of leaves.

\begin{remark}
Compare the edge-insertion procedure with the coproduct defined in the previous section, where cutting a single edge in a tree collapses two of the remaining edges into one. This already hints to the duality of these two procedures.
\end{remark}

\begin{remark} \label{ghost}
We can alternatively view our rooted trees as having an additional edge at the root. Under this interpretation, insertion involves splitting an edge of $T$ and identifying the new vertex with the upper vertex of the root of $S$. Then the sum $T \ins S$ is indexed by all edges of $T$, since the term $\fM(T,S)$ is obtained by inserting $S$ at the root edge of $T$.

\begin{align*}
\begin{tikzpicture}
  [baseline=-3mm,
   level distance=3.5mm,
   level 1/.style={level distance=4.5mm},
   sibling distance=13mm,
   every node/.style={text height=0.5mm}]
  \coordinate
     child {
       child {
         child {node {$\alpha$}}
         child {node {$\beta$}}
       }
       child {node {$\gamma$}}
       edge from parent node[left] {$e$}
     };
\end{tikzpicture}
  \ins_e
  \Tree [ [ $\delta$ $\epsilon$ ] ]
  &=
  \Tree [ [ [ [ $\alpha$ $\beta$ ] $\gamma$ ] [ $\delta$ $\epsilon$ ] ] ]
\end{align*}

However, the source of the root edge will never be viewed as a leaf. For this reason this edge will be considered a \emph{ghost} or \emph{virtual edge}.
\end{remark}
We will primarily use the first model in order to follow the conventions of the existing literature, but we will refer to the second model when it is more convenient for proofs, such as the following.

\begin{theorem}
  $\ins$ is a pre-Lie operation.
\end{theorem}

\begin{proof}
  Let $T_1, T_2, T_3$ be three trees. All trees are viewed as having a root edge. We will assign a geometric meaning to the associator,
  \[ A(T_1,T_2,T_3) = (T_1 \ins T_2) \ins T_3 - T_1 \ins (T_2 \ins T_3), \]
  which will show that the associator is unchanged when the positions of $T_2$ and $T_3$ are swapped.
  
  Note that $T_2 \ins T_3$ contains the term $\fM(T_2,T_3)$: the terms in $T_1 \ins (T_2 \ins T_3)$ resulting from $T_1 \ins \fM(T_2,T_3)$ correspond exactly to the terms in $(T_1 \ins T_2) \ins T_3$ where $T_2$ is inserted into an edge of $T_1$, and then $T_3$ is inserted into the edge which was originally the root of $T_2$.

  The terms in the associator $A(T_1,T_2,T_3)$ consist of the trees obtained when $T_2$ and $T_3$ either are inserted into distinct edges of $T_1$, or are both inserted into the same edge. In the latter case, each edge produces two terms: upon inserting $T_2$ into $T_1$ at an edge $e$, the edge is split, so $T_3$ is subsequently inserted into either the upper or lower edge, so (the root of) $T_3$ appears either above or below (the root of) $T_2$.

  The exact same terms appear in $A(T_1,T_2,T_3)$, since the order in which $T_2$ and $T_3$ are inserted is ultimately irrelevant. Hence $\ins$ satisfies the Vinberg identity, making it a pre-Lie operator.
\end{proof}

\begin{remark}
$\cL$ is not free as a pre-Lie algebra, since
\[ \bintree{[[][]]} \ins \bullet = 3 \bintree{[[[][]][]]} \]
and
\[ \bullet \ins \bintree{[[][]]} = \bintree{[[[][]][]]}. \]
\end{remark}

\subsection{An operadic perspective}

Since all rooted nonplanar binary trees are elements of the Merge operad $\cM$, we can write a decomposition of the $\ins$ operator as a series of coloured operadic insertions. One component of this decomposition is $ \Tree [ $\ell_1$ $\ell_2$ ] \in \cM(2)$, which serves as an operadic representation of the Merge operation. Crucially, it provides additional evidence that the term $\fM(T,S)$ should naturally appear in $T \ins S$.

Let $T \in \cM(n)$ and $S \in \cM(m)$. Consider $T \ins_e S$, the term of $T \ins S$ where $S$ is inserted into the edge $e$ of $T$. (For the term $\fM(T,S)$, this is the imaginary edge above the root.) Let $T' \in \cM(n')$ be the subtree starting from the bottom endpoint of $e$ (if the term is $\fM(T,S)$, then $T' = T$).

Let $T'' = T /^c T' \in \cM(n - n' + 1)$ be the quotient of $T$ where all of $T'$ is contracted to a single leaf with label $\ell$. Indeed, $T = T'' \circ_\ell T'$. Then
\[ T \ins_e S = T'' \circ_\ell ( \Tree [ $\ell_1$ $\ell_2$ ] )  \circ_{\ell_1} T' \circ_{\ell_2} S = T'' \circ_\ell \fM(T', S). \]

Indeed, if $e$ is the imaginary edge above the root, then $T' = T$ (as mentioned above) and $T'' = \ell \in \cM(1)$ is the single leaf with label $\ell$, which acts as the identity with respect to the coloured composition $\circ_\ell$. Then \[ T \ins_e S = \ell \circ_\ell \fM(T, S) = \fM(T,S). \]

\begin{example}
Let $T = \Tree[ [ $\alpha$ $\beta$ ] $\gamma$ ]$ and $S = \Tree [ $\delta$ $\epsilon$ ]$. We illustrate the insertion
\begin{align*}
\begin{tikzpicture}
  [baseline=-5mm,
   level distance=3.5mm,
   sibling distance=13mm,
   every node/.style={text height=0.5mm}]
  \coordinate
       child {
         child {node {$\alpha$}}
         child {node {$\beta$}}
         edge from parent node[left] {$e$}
       }
       child {node {$\gamma$}}
       ;
\end{tikzpicture}
  \ins_e
  \Tree [ $\delta$ $\epsilon$ ]
  &=
  \Tree [ [ [ $\alpha$ $\beta$ ] [ $\delta$ $\epsilon$ ] ] $\gamma$ ]
\end{align*}
as an operadic composition. Then
\[ T' = \Tree [ $\alpha$ $\beta$ ], \quad T'' = \Tree [ $\ell$ $\gamma$ ], \]
so
\[ T'' \circ_\ell \fM(T', S) = \Tree [ $\ell$ $\gamma$ ] \circ_\ell \Tree [ [ $\alpha$ $\beta$ ] [ $\delta$ $\epsilon$ ] ] = \Tree [ [ [ $\alpha$ $\beta$ ] [ $\delta$ $\epsilon$ ] ] $\gamma$ ] = T \ins_e S. \]
\end{example}

\subsection{The Guin-Oudom construction}

Guin and Oudom \cite{oudom2008lie} provide an explicit description of the universal enveloping algebra of a Lie algebra constructed from a pre-Lie operation, which is a Hopf algebra. We will repeat this construction and recall that the Connes-Kreimer Hopf algebra is the dual to a Hopf algebra produced from this procedure, where the notion of duality is based on a specific inner product described later in Section \ref{duality}.

Another exposition of the Guin-Oudom procedure and how it relates to the Lie-algebraic construction of the Connes-Kreimer Hopf algebra is found in \cite{foissy2025operads}.

Let $(\cL,\ins)$ be a pre-Lie algebra and let $\Sym(\cL)$ be the symmetric algebra of commuting polynomials over $\cL$. Then $\Sym(\cL)$ is naturally equipped with the shuffle coproduct, which is defined by
\[ \Sha(x_1 x_2 \dots x_n) = \sum_{I \subseteq [n]} x_I \otimes x_{[n] \setminus I} \]
where each $x_i \in \cL$ and, given an index set $J = \{j_1, \dots, j_k\}$, we write $x_J = x_{j_1} \cdots x_{j_k}$. The shuffle coproduct is coassociative and cocommutative.

The pre-Lie operation $\ins$ extends to an operation on $\Sym(\cL)$, which satisfies the following properties

\begin{enumerate}
  \item $A \ins 1 = A$
  \item $X \ins A Y = (X \ins A) \ins Y - X \ins (A \ins Y)$
  \item $AB \ins C = \sum_{(C)} (A \ins C_{(1)}) (B \ins C_{(2)})$
\end{enumerate}
for elements $X, Y \in \cL$ and $A,B,C \in \Sym(\cL)$. Recursively, this uniquely defines $A \ins B$ for any two monomials $A,B \in \Sym(\cL)$ and extends to all of $\Sym(\cL)$ by linearity.

We remark that if we consider property (2) with an element $Z \in \cL$, then
\[ X \in ZY = (X \ins Z) \ins Y - X \ins (Z \ins Y) \]
is equal to the associator $A(X,Z,Y)$ of $X,Z,Y$ with respect to $\ins$. Since $\cL$ is a pre-Lie algebra, the Vinberg identity $A(X,Z,Y) = A(X,Y,Z)$ holds.

The product of $\cU(\cL)$ as a Hopf algebra is not $\ins$, but rather a further modification of it.

We form another product $\ast$ on $\Sym(\cL)$ by first taking the extension of $\ins$ to $\Sym(\cL)$ and then defining 
\[ A \ast B := \sum_{(B)} (A \ins B_{(1)} ) B_{(2)} \]
or monomials $A,B$, where the sum is indexed over terms of the shuffle coproduct (using Sweedler notation). Guin and Oudom prove that $\ast$ is associative.

The difference between $\ast$ and $\ins$ is that $\ins$ represents the simultaneous insertion of \emph{all} factors in the monomial $B$ into $A$. That is, $A \ins B$ contains only terms where all trees in $B$ are simultaneously inserted into the forest $A$, meaning that every tree in $B$ becomes part of a tree in $A$. Meanwhile, $A \ast B$ allows for terms where only some of the trees in $B$ are inserted into $A$; the unused trees of $B$ remain as additional elements in the resulting forest.

\begin{example}
Let $T = \Tree [ $\alpha$ $\beta$ ]$ and $F = S_1 \sqcup S_2$, where $S_1, S_2$ are trees. Then

\begin{align*}
  T \ins F
  =
  \Tree [ [ $\alpha$ $\trinum{S_1}$ ] [ $\trinum{S_2}$ $\beta$ ] ]
  &+
  \Tree [ [ [ $\alpha$ $\trinum{S_1}$ ] $\trinum{S_2}$ ] $\beta$ ]
  +
  \Tree [ [ [ $\alpha$ $\trinum{S_2}$ ] $\trinum{S_1}$ ] $\beta$ ] \\
  &+ \Tree [ $\alpha$  [$\trinum{S_1}$ [ $\trinum{S_2}$  $\beta$ ] ] ]
  +
  \Tree [ $\alpha$  [$\trinum{S_2}$ [ $\trinum{S_1}$  $\beta$ ] ] ]
  +
  \Tree [ [ $\alpha$ $\trinum{S_2}$ ] [ $\trinum{S_1}$ $\beta$ ] ]
\end{align*}
while
\begin{align*}
  T \ast F  = T \ins F
  &+ \Tree [ [ $\alpha$ $\trinum{S_1}$ ] $\beta$ ] \sqcup S_2
   + \Tree [ $\alpha$ [ $\trinum{S_1}$  $\beta$ ] ] \sqcup S_2
   + \Tree [ [ $\alpha$  $\beta$ ] $\trinum{S_1}$ ] \sqcup S_2 \\
  &+ \Tree [ [ $\alpha$ $\trinum{S_2}$ ] $\beta$ ] \sqcup S_1
   + \Tree [ $\alpha$ [ $\trinum{S_2}$  $\beta$ ] ] \sqcup S_1
   + \Tree [ [ $\alpha$  $\beta$ ] $\trinum{S_2}$ ] \sqcup S_1 \\
   &+ T \sqcup S_1 \sqcup S_2.
\end{align*}
\end{example}

This equips $(\Sym(\cL), \ast, \Sha)$ with the structure of a noncommutative and cocommutative Hopf algebra, which is isomorphic to the universal enveloping algebra of $\cL$. We will refer to this Hopf algebra as $\cU(\cL) = (\cU(\cL), \ast, \Sha)$, as this is the only presentation of the universal enveloping algebra that we will explicitly use.

\section{Duality of Hopf algebras} \label{duality}

The notion of duality we is based on the inner product on the polynomial algebra of trees defined in Hoffman \cite{hoffman2003combinatorics}, which involves a symmetry coefficient. It turns out the same pairing restricted to binary trees also induces a duality between $\cH$, the Hopf algebra of subsets of leaves, and $\cU(\cL)$, the universal enveloping algebra induced by the edge-insertion pre-Lie algebra. We owe much of the notation here to Hoffman.

Hoffman showed that the Connes-Kreimer and Grossman-Larson Hopf algebras are dual to each other with respect to this pairing, correcting an initial attempt by Panaite to do the same \cite{panaite2000relating}.

\subsection{Symmetries of forests}

 Let $SG(F)$ be the group of automorphisms on $F$, which are graph isomorphisms that map roots to roots and that preserve implicit the directed graph structure induced by the root node. 
 
 We write $s_F = |SG(F)|$ for the cardinality of the automorphism group.

Each graph automorphism induces an automorphism on the edges of $F$.

If $F = T_1^{a_1} \cdots T_k^{a_k}$, where $T_1, \dots, T_k$ are all distinct trees and each $a_i$ is a positive integer exponent, then
\[ s_F = a_1 ! \cdots a_k! \cdot s_{T_1}^{a_1} \cdots s_{T_k}^{a_k}, \]
which follows directly from the identification
\[ SG(F) \cong \fS_{a_1} \oplus \cdots \oplus \fS_{a_k} \oplus SG(T_1)^{\oplus a_1} \oplus \cdots \oplus SG(T_k)^{\oplus a_k}. \]

When $T$ is a binary tree, the only automorphisms of $T$ involve swapping the children of interior nodes which have a pair of identical children. Hence, for binary trees $T$,
\[ s_T = 2^r, \]
where $r$ is the number of interior nodes with identical children: at each such node, the map either swaps its twin children or it does not. 

\begin{remark}
Recall from Remark~\ref{ghost} that we can view our binary trees as having an additional virtual edge at the root. The virtual edge does not affect the automorphism group because it is fixed by all automorphisms. Thus we can refer to $SG(F)$ unambiguously, whether or not we consider trees to have a virtual root edge.
\end{remark}

\begin{example}
We provide the values of $s_T$ for some small trees with labelled leaves below:

\begin{center}
\begin{tabular}{|c||c|c|c|c|c|c|c|c|c|}
\hline
$T$ & $\alpha$
    & $\CoII{$\alpha$}{$\beta$}$
    & $\CoII{$\alpha$}{$\alpha$}$
    & $\CoIII{$\alpha$}{$\beta$}{$\gamma$}$
    & $\CoIII{$\alpha$}{$\alpha$}{$\beta$}$
    & \Tree [ [ $\alpha$ $\beta$ ] [$\alpha$ $\beta$ ] ] 
    & \Tree [ [ $\alpha$ $\alpha$ ] [$\beta$ $\beta$ ] ] 
    & \Tree [ [ $\alpha$ $\alpha$ ] [$\alpha$ $\alpha$ ] ] 
    \\ \hline\hline
$s_T$ & 1 & 1 & 2 & 1 & 2 & 2 & 4 & 8 \\ \hline
\end{tabular}
\end{center}
\end{example}

\subsection{Defining the inner product}
Let $k[\cT]$ denote the polynomial algebra of nonplanar binary trees. This algebra has a linear basis of forests. The pairing is defined as
\[ \langle F, F' \rangle := s_F \delta_{F,F'} \]
for any forests $F,F' \in k[\cT]$.

We will now view the above pairing on $k[\cT]$ as a map
\[ \langle -, - \rangle :  \cU(\cL) \times \cH \to k. \]
The product of $\cH$, denoted by $\cdot_\Sha$, is defined by the identity
\[ \langle \Sha(F_1), F_2 \otimes F_3 \rangle = \langle F_1, F_2 \cdot_\Sha F_3 \rangle. \]
Since $\Sha$ is cocommutative, $\cdot_\Sha$ is commutative. The coproduct of $\cH$ is defined by
\[ \langle F_1 \ast F_2, F_3 \rangle = \langle F_2 \otimes F_1, \Delta_\ast(F_3) \rangle; \]
note that the positions of $F_2$ and $F_1$ are swapped on the right-hand side.

This makes $(\cH, \cdot_\Sha, \Delta_\ast)$ a Hopf algebra which is dual to the opposite of $(\cU(\cL), \ast, \Sha)$. It remains to show that $\cdot_\Sha$ is the disjoint union $\sqcup$, while $\Delta_\ast$ is the same coproduct $\Delta$ that was defined in Section \ref{newhopf}.

\subsection{Showing that $\cdot_\Sha = \sqcup$} Let $F_1 = T_1^{a_1} \cdots T_{k}^{a_k}$ be a forest, with $T_1, \dots, T_k$ being a collection of distinct tree classes. Let $F_2 = T_1^{b_1} \cdots T_k^{b_k}$ be such that $b_i \leq a_i$ and let $F_3 = T_1^{a_1 - b_1} \cdots T_k^{a_k - b_k}$. Then $F_2 \otimes F_3$ appears in $\Sha(F)$ with coefficient
\[ \prod_{i=1}^k {a_i \choose b_i} = \prod_{i=1}^k \frac{a_i!}{b_i! (a_i - b_i)!} \]
and so
\begin{align*}
  \langle F_1, F_2 \cdot_\Sha F_3 \rangle
  &= \langle \Sha(F_1), F_2 \otimes F_3 \rangle \\
  &= s_{F_2}  s_{F_3}   \Big( \prod_{i=1}^k \frac{a_i!}{b_i! (a_i - b_i)!} \Big) \\
  &= \Big( \prod_{i=1}^k b_i!  s_{T_i}^{b_i} \Big) \Big( \prod_{i=1}^k (a_i - b_i)!  s_{T_i}^{a_i - b_i} \Big) \Big( \prod_{i=1}^k \frac{a_i!}{b_i! (a_i - b_i)!} \Big) \\
  &= \prod_{i=1}^k a_i!  s_{T_i}^{a_i} \\
  &= s_{F_1} \\
  &= \langle F_1, F_2 \sqcup F_3 \rangle.
\end{align*}
 Hence $\cdot_\Sha = \sqcup$.

\subsection{Showing that $\Delta = \Delta_*$}
Our goal is to show that
\[ \langle F_1 \ast F_2, F_3 \rangle = \langle F_2 \otimes F_1, \Delta(F_3) \rangle. \]

In the following proofs, it will be convenient to view all trees as having a virtual root edge. When speaking of the edge set $E(F)$, the root edge of each tree in $F$ is included; however, its source is not considered a leaf.

\subsubsection{Enumeration of cuts}

Given a forest $F \in k\{\cT\}$, let $\cC(F)$ be the collection of binary-admissible cuts on $F$. Each $C \in \cC(F)$ is a subset of $E(F)$ such that
\begin{enumerate}
  \item no pair of edges in $C$ appear on the same path from root to leaf of any tree in $F$;
  \item no two edges in $C$ share the same parent vertex.
\end{enumerate}
This is the natural way to extend the notion of binary-admissible cuts on individual trees to forests using multiplicativity of $\Delta$.

The coproduct of $F$ is a sum over all cuts in $\cC(F)$:
\[ \Delta(F) = \sum_{C \in \cC(F)} P^C(F) \otimes R^C(T) \]
where $P^C(F)$ is the forest obtained by taking all the trees below each deleted edge in $C$, and $R^C(T)$ is the binary forest containing all of the rooted trees above the deleted edges.

Note that if the root edge of a tree $T$ in $F$ lies in $C$, then $T$ appears in $P^C(F)$ and has no remainder in $R^C(F)$. For instance, the term $F \otimes 1$ corresponds to the cut where all root edges are selected.

\subsubsection{Enumeration of graphs}

Given two forests $F_1, F_2 \in k\{\cT\}$, let $\cG(F_1, F_2)$ be the collection of distinct grafts of $F_1$ into $F_2$. When speaking of grafts we include those terms introduced by the extension of $\ins$ on $\Sym(\cL)$ to $\ins$, so certain trees in $F_1$ may not be inserted into any edge in $F_2$, remaining instead as a term in the resulting forest. Each graft $G$ contributes a single forest $F_2 \ast_G F_1$ to $F_2 \ast F_1$, so that
\[ F_2 \ast F_1 = \sum_{G \in \cG(F_1, F_2)} F_2 \ast_G F_1.\]
Each graft consists of the following data:
\begin{enumerate}
  \item a map $\mu_G$ from the trees in $F_1$ to $E(F_2) \sqcup \{e_0\}$, where $e_0$ is a dummy element which is the image of any tree not inserted into an edge of $F_2$;
  \item a linear order on $\mu_G^{-1}(e)$ for every $e \in E(F_2)$.
\end{enumerate}
No linear order is imposed on the preimage of $e_0$.

\subsubsection{Multiplicities of cuts and grafts}

Let $F_1, F_2$ and $F_3$ be forests. Following the notation of Hoffman, we define 
\begin{align*}
  n(F_1, F_2; F_3)  &= | \{ G \in \cG(F_1, F_2) : F_2 \ast_G F_1 = F_3 \}|  \\
  &= \text{the coefficient of $F_3$ in $F_2 \ast F_1$}
\end{align*}
and
\begin{align*}
  m(F_1, F_2; F_3) &= | \{ C \in \cC(F_3) : P^C(F_3) = F_1 \text{ and } R^C(F_3) = F_2 \} | \\
  &= \text{the coefficient of $F_1 \otimes F_2$ in $\Delta(F_3)$},
\end{align*}
so that
\[ \langle F_1 \ast F_2, F_3 \rangle = n(F_1, F_2; F_3) \langle F_3, F_3 \rangle = n(F_1, F_2; F_3) s_{F_3} \]
and
\begin{align*}
  \langle F_2 \otimes F_1, \Delta(F_3) \rangle
  &= m(F_1, F_2; F_3) \langle F_2 \otimes F_1, F_2 \otimes F_1 \rangle \\
  &= m(F_1, F_2; F_3) \langle F_2, F_2 \rangle \langle F_1, F_1 \rangle
   = m(F_1, F_2; F_3) s_{F_1} s_{F_2}.
\end{align*}

\begin{theorem}
  Let $F_1, F_2$ and $F_3$ be three forests. Then
  \[ \langle F_1 \ast F_2, F_3   \rangle = \langle F_2 \otimes F_1, \Delta(F_3) \rangle, \]
  so $\Delta = \Delta_*$.
\end{theorem}

\begin{proof}
With $F_1, F_2, F_3$ fixed, our goal is to show that
\[ n(F_1, F_2; F_3) s_{F_3} = m(F_1, F_2; F_3) s_{F_1} s_{F_2}. \]
Define
\[ \cC' = \{  C \in \cC(F_3) : P^C(F_3) = F_1 \text{ and } R^C(F_3) = F_2 \} \] 
so that $n(F_1, F_2; F_3) = |\cC'|$.

Consider any cut $C \in \cC'$, viewing it as a subset of $E(F_3)$. Take an automorphism $\sigma \in SG(F_3)$ and let $SG(F_3)$ act on $2^{E(F_3)}$. Let $\Orb_\cC(C, F_3)$ be the orbit of $C \in 2^{E(F_3)}$ under the action of $SG(F_3)$. Then $\sigma(C) \in \cC'$, so $\Orb_\cC(C, F_3) \subseteq \cC'$.

However, there may be other cuts in $\cC'$ which are not in the orbit of $C$. Thus $\cC'$ has a decomposition as a disjoint union of \emph{cut-classes}, each of which is the orbit of any cut in its class:
\[ \cC' = \Orb_\cC(C_1, F_3) \sqcup \cdots \sqcup \Orb_\cC(C_k, F_3) \]
where each $C_i$ is a representative of a distinct cut-class.

With the same forests $F_1, F_2, F_3$ as above, we define
\[ \cG' = \{ G \in \cG(F_1, F_2) : F_2 \ast_G F_1 = F_3 \} \]
so that $m(F_1, F_2; F_3) = |\cG'|$, and we partition $\cG'$ into graft-classes.

We will abstract some of the data in a graft. Instead of a linear order on the trees in $\mu_G^{-1}(e)$ for each edge in $E(F_2)$, the linear order now only considers the isomorphism classes of each tree. Now each $G \in \cG'$ consists of an edge-set $E_G \subseteq E(F_3) \sqcup \{e_0\}$, along with a decoration $L_G(e)$, a finite sequence of tree-isomorphism classes, for each $e \in G \setminus \{e_0\}$. For each edge $e$ in $F_2$ to which some trees from $F_1$ is grafted, we list the trees from top to bottom, starting from the source of $e$ to its target. We write $G = (E_G, L_G)$. 

Now let $\sigma \in SG(F_2)$. We allow $SG(F_2)$ to act on grafts by defining, for a graft $G$,
\[ \sigma(G) = (E_{\sigma(G)}, L_{\sigma(G)}) \]
where $E_{\sigma(G)} = \sigma(E_G)$ and $L_{\sigma(G)}(\sigma(e)) = L_G(e)$ for each $e \in E_G$. Now an automorphism also carries the decoration on each edge to its image.

Let $\Orb_\cG(G, F_2)$ denote the orbit of $G$ under the action of $SG(F_2)$ on grafts.  Let $F_1 = T_1^{a_i} \cdots T_r^{a_r}$, where each $T_i$ is represents a distinct isomorphism class of trees. Let $\cG''$ be the disjoint union of these orbits, known as graft-classes:
\[ \cG'' =  \Orb_\cG(G_1, F_2) \sqcup \cdots \sqcup \Orb_\cG(G_\ell, F_2) \]

Finally, if $F_1 = T_1^{a_i} \cdots T_r^{a_r}$, where each $T_i$ represents a distinct isomorphism class of tree, then $\cG'$ is a disjoint union of \[ a_1! a_2! \cdots a_r! = \frac{s_{F_1}}{s_{T_1}^{a_1} \cdots s_{T_r}^{a_r}} \] copies of $\cG''$; each copy amounts to an explicit choice of trees in each decoration.

Now we claim that there is a bijection between cut-classes and graft-classes. We construct two maps: one from cut-classes to graft-classes and one from graft-classes to cut-classes.

\begin{itemize}
 \item  Let $C$ be a representative of a particular cut-class in $\cC'$. If there is a path subgraph $P$ on $C$ such that the sibling of any edge of $P$ other than the first and last are severed, then $P$ becomes condensed to a single edge in $F_2 = R^C(F_3)$. We extract all path subgraphs of $C$ fitting this description and match each subgraph to the remaining edge in $F_2$. If the root of any tree in $F_3$ is severed, then that edge is mapped to the virtual edge $e_0$. This uniquely determines a graft in $\cG''$.

 \item  On the other hand, let $G$ represent of a graft-class in $\cG''$. Let $C$ be the cut of $F_3$ where all the roots of grafted trees in $F_1$ are severed; this uniquely determines a cut-class in $\cC'$.
\end{itemize}
  
  These two assignments are inverse to each other.

Now that we know that the cut-classes and graft-classes are in bijection, we can re-index the representatives so that $\Orb_\cC(C_i, F_2)$ is the cut-class corresponding to the graft-class $\Orb_\cG(G_i, F_2)$.

Since each of $\cC'$ and $\cG'$ is a disjoint union of orbits, their cardinalities are given by the sums
\[ |\cC'| = \sum_{i=1}^k |\Orb_\cC(C_i, F_3)|, \quad |\cG'| = \sum_{i=1}^k a_1! \cdots a_r! |\Orb_\cG(C_i, F_2)|. \]
The orbit-stabilizer theorem states that
\[ |\Orb_\cC(C_i, F_3)| = \frac{s_{F_3}}{|\Stab_\cC(C_i, F_3)|} \]
and
\[ |\Orb_\cG(G_i, F_2)| =  \frac{s_{F_2}}{|\Stab_\cG(G_i, F_2)|}, \]
where $\Stab_\cC(C_i, F_3)$ is the subgroup of $SG(F_3)$ fixing the edges in $C_i$ and $\Stab_\cG(G_i, F_2)$ fixes the edges in $E_{G_i}$ (which also fixes all their decorations).

Let $\fF(C_i) \subseteq E(F_3)$ be the collection of all edges in $F_3$ appearing below the edges in $C_i$. Recall that $F_1 = T_1^{a_i} \cdots T_r^{a_r}$, where each $T_i$ is represents a distinct isomorphism class of trees. Then
\[ \Stab_\cC(C_i, F_3) \cong \Stab_\cC(\fF(C_i), F_3) \oplus SG(T_1)^{\oplus a_1} \oplus \cdots \oplus SG(T_r)^{\oplus a_r}, \]
so
\[ |\Orb_\cC(C_i, F_3)| = \frac{s_{F_3}}{|\Stab_\cC(\fF(C_i), F_3)| \cdot s_{T_1}^{a_1} \cdots s_{T_r}^{a_r}} \]
and so
\[ |\cC'| = \sum_{i=1}^k |\Orb_\cC(C_i, F_3)| = \frac{s_{F_3}}{s_{T_1}^{a_1} \cdots s_{T_r}^{a_r}} \sum_{i=1}^k \frac{1}{|\Stab_\cC(\fF(C_i), F_3)}  \]
while
\[ |\cG'| = a_1! \cdots a_r! s_{F_2} \sum_{i=1}^k \frac{1}{|\Stab_\cG(G_i, F_2)|} \]
But there is an identification of $\Stab_\cC(\fF(C_i), F_3)$ with $\Stab_\cG(G_i, F_2)$ for each $i$, so the two summations are equal and
\[ |\cC'| \frac{ s_{T_1}^{a_1} \cdots s_{T_r}^{a_r} }{s_{F_3}} = |\cG'|\frac{1 }{ a_1! \cdots a_r! s_{F_2}  } \]
Equivalently,
\[  |\cC'| \cdot (a_1! \cdots a_r! \cdot  s_{T_1}^{a_1} \cdots s_{T_r}^{a_r}) \cdot s_{F_2} =  |\cG'| \cdot s_{F_3}\]
hence
\[ n(F_1, F_2; F_3) s_{F_1} s_{F_2}  = m(F_1, F_2; F_3) s_{F_3}. \]
\end{proof}

\begin{remark}
  The fact that $\cC'$ decomposes into a union of cut-classes is the key difference between the edge-insertion and vertex-insertion pre-Lie algebras. When considering vertex-insertion, each $\cC'$ is the orbit of a single cut.

  Similarly, compare the partitioning of $\cG''$ to the case of the vertex-insertion Lie algebra. A vertex-graft does not need to keep track of any linear order on the inserted trees, since all trees are non-planar. Simultaneous insertion of trees at a vertex is independent of any ordering on the trees.
\end{remark}

\forestset{
    bcut/.style={edge path={\noexpand\path[\forestoption{edge}](!u.parent anchor)-- node[red] {\scalebox{2}{$\diagdown$}} (.child anchor)\forestoption{edge label};}},
    fcut/.style={edge path={\noexpand\path[\forestoption{edge}](!u.parent anchor)-- node[red] {\scalebox{2}{$\diagup$}} (.child anchor)\forestoption{edge label};}},
}

\begin{example}
This example illustrates the notion of cut-classes and graft-classes used in the previous proof. Let
\[ F_1 = \bintree{[[][]]}, \quad F_2 = \bullet \bullet \bintree{[[][]]}, \quad F_3 =  \bintree{[[[[[][]][]][[][]]][]]}, \]
which have symmetry coefficients
\[ s_{F_1} = 2, s_{F_2} = 4, s_{F_3} = 4.  \]

\begin{center}
\begin{tabular}{ |c||c| }
     \hline
Cut-classes in $\cC'$ & Graft-classes in $\cG''$ \\
     \hline\hline
\(
\bintree{[[[[[,bcut[][]][]][[,bcut][]]][,fcut]]]},
\bintree{[[[[[,bcut[][]][]][[][,fcut]]][,fcut]]]}
\)

&

\( \bigbintree{[
[,edge={blue,very thick},edge label={node[blue,midway,right,font=\scriptsize]{ $(\bullet)$ }}
[,edge={blue,very thick},edge label={node[blue,midway,left,font=\scriptsize]{ $(\wedge)$ }}]
[,edge={blue,very thick},edge label={node[blue,midway,right,font=\scriptsize]{ $(\bullet)$ }}
]
]]},
 \bigbintree{[
[,edge={blue,very thick},edge label={node[blue,midway,right,font=\scriptsize]{ $(\bullet)$ }}
[,edge={blue,very thick},edge label={node[blue,midway,left,font=\scriptsize]{ $(\bullet)$ }}]
[,edge={blue,very thick},edge label={node[blue,midway,right,font=\scriptsize]{ $(\wedge)$ }}
]
]]} \) \\ \hline

\(
\bintree{[[[[[[][,fcut]][]][,fcut[][]]][,fcut]]]},
\bintree{[[[[[[,bcut][]][]][,fcut[][]]][,fcut]]]}
\)

&

\( \bigbintree{[[,edge={blue,very thick},edge label={node[blue,midway,right,font=\scriptsize]{ $(\bullet, \wedge)$ }}[,edge={blue,very thick},edge label={node[blue,midway,left,font=\scriptsize]{ $(\bullet)$ }}][]]]}
,

\bigbintree{[[,edge={blue,very thick},edge label={node[blue,midway,right,font=\scriptsize]{ $(\bullet, \wedge)$ }}[][,edge={blue,very thick},edge label={node[blue,midway,right,font=\scriptsize]{ $(\bullet)$ }}]]]} \)

\\ \hline

\(
\bintree{[[[[[[][]][,fcut]][,fcut[][]]][,fcut]]]}
\)

&

\( \bigbintree{[[,edge={blue,very thick},edge label={node[blue,midway,right,font=\scriptsize]{ $(\bullet, \wedge, \bullet)$ }}[][]]]} \) \\ \hline

\(
\bintree{[[[[[[][,fcut]][,fcut]][,fcut[][]]][]]]},
\bintree{[[[[[[,bcut][]][,fcut]][,fcut[][]]][]]]}
\)

&

\( \bigbintree{[[[,edge={blue,very thick},edge label={node[blue,midway,left,font=\scriptsize]{ $(\wedge, \bullet, \bullet)$ }}][]]]},
\bigbintree{[[[][,edge={blue,very thick},edge label={node[blue,midway,right,font=\scriptsize]{ $(\wedge, \bullet, \bullet)$ }}]]]}\)
\\ \hline
\end{tabular}
\end{center}

Since $F_2$ contains two copies of $\bullet$ and one copy of $\wedge$, there are $2! \cdot 1! = 2$ disjoint copies of $\cG''$ in $\cG'$. So $|\cC'| = 7$ and $|\cG'| = 2! \cdot |\cG''| = 14$. Hence,
\[ |\cC'| s_{F_1} s_{F_2} = 7 \cdot 2 \cdot 4 = 56 = 14 \cdot 4 = |\cG'| s_{F_3}. \]
\end{example}

\section{The pre-Lie exponential} \label{consequences}

We will specialize the multiplicities $n(F_1, F_2; F_3)$ and $m(F_1, F_2; F_3)$ to the case where $F_2 = T$ is a tree and $F_1 = \bullet$. Following Hoffman's notation, define the growth operator $\fN : k\{\cT_n\} \to k \{\cT_{n+1}\}$ by
\[ \fN(T) = T \ins \bullet = \sum_{T' \in \cT} n(\bullet, T; T') T' \]
and the pruning operator $\fP : k\{\cT_n\} \to k\{\cT_{n-1}\}$ by
\[ \fP(T) = \sum_{T' \in \cT} m(\bullet, T'; T) T'. \]
Note that both operators are written as a sum over all isomorphism classes of trees, regardless of degree. This is reasonable because $n(\bullet, T; T') = 0$ whenever $|T| \neq |T'| - 1$ and $m(\bullet, T'; T) = 0$ whenever $|T| \neq |T'| + 1$, so $\fN(T)$ and $\fP(T)$ are finite and homogeneous.

This indexing also does not produce any terms which are a forest of more than one tree, since we work with the original $\ins$ operator instead of its extension to $\ast$.

\begin{proposition}
  The operators $\fN$ and $\fP$ are adjoint on $k\{\cT\}$ with respect to the pairing $\langle -, - \rangle$.
\end{proposition}

\begin{proof}
  It is enough to check that $\langle \fN(T), T' \rangle = \langle T, \fP(T') \rangle$. Indeed,
  \[  \langle \fN(T), T' \rangle = \langle \sum_{T'' \in \cT} n(\bullet, T; T'') T'', T' \rangle = n(\bullet, T; T') \langle T', T' \rangle = n(\bullet, T; T') s_{T} \] 
  and
  \[  \langle T, \fP(T') \rangle = \langle T, \sum_{T'' \in \cT} m(\bullet, T''; T') T'' \rangle = m(\bullet, T; T') \langle T, T\rangle = m(\bullet, T; T') s_{T'}, \]
  and $n(\bullet, T; T') s_{T} = m(\bullet, T; T') s_{T'}$.
\end{proof}

As a corollary, $\fN^k$ and $\fP^k$ are adjoint for any $k \geq 1$.

\begin{lemma}
  Let $k \geq 2$. If $T \in \cT_k$, then $\fP^{k-1}(T) = k! \cdot \bullet$.
\end{lemma}

An intuitive way to see this is that the coefficient of $\bullet$ in $\fP^{k-1}(T)$ is the number of ways to reduce $T$ to a single leaf by progressively removing one leaf at a time, assuming that all leaves are distinct. This is formalized in an inductive proof.
  
\begin{proof}

  For the base case, consider $k=2$. The only unlabelled tree of degree $2$ is $\bintree{[[][]]}$ and
  \[ \fP( \bintree{[[][]]} ) = 2 \bullet = 2! \cdot \bullet. \]
  Now let $k > 2$ and suppose that $\fP^{k-2}(T') = (k-1)! \cdot \bullet$ for all $T' \in \cT_{k-1}$. Then $\fP(T)$ is a sum of $k$ trees $T_1, \dots, T_{k} \in \cT_{k-1}$, not necessarily distinct, each of which corresponds to a choice of leaf in $T$ to sever. Hence
  \[ \fP^{k-1}(T) = \fP^{k-2}(\fP(T) )
  = \fP^{k-2}\Big( \sum_{i=1}^k T_i  \Big)
  = \sum_{i=1}^k \fP^{k-2}(T_i) = \sum_{i=1}^k (k-1)! \cdot \bullet =  k! \cdot \bullet. \]
\end{proof}

The pre-Lie exponential of an element $x \in \cL$ is defined as the formal series
\[ W(x) = x + \frac{1}{2!} x \ins x + \frac{1}{3!} (x \ins x) \ins x + \cdots. \]
When equipped with a multiplication operation
\[ W(x) \star W(y) = W(C(x,y)) \]
in terms of the Baker-Campbell-Hausdorff formula
\[ C(x,y) = x + y + \frac{1}{2} [x,y] + \frac{1}{12} ( [x,[x,y]] + [y,[y,x]] ) - \frac{1}{24} [y,[x,[x,y]]] + \cdots, \]
the formal series $W(x)$ have the structure of a multiplicative group. See \cite{marcolli2015graph} for more references and Section 3.6 of \cite{CartierPatras} for the formula of the full expansion of the Baker-Campbell-Hausdorff formula.

Let $\bullet$ be an unlabelled single-vertex tree in the pre-Lie algebra of edge-insertion. Then $W(\bullet)$ is a formal sum over repeated applications of the growth operator $\fN$ to $\bullet$,
\[ W(\bullet) = \sum_{k=1}^\infty \frac{1}{k!} \fN^{k-1}(\bullet), \]
where $\fN^0(\bullet) = \bullet$.

Since each isomorphism class of unlabelled trees with $k$ leaves appears as a term of $\fN^{k-1}(\bullet)$, each linear generator has a positive coefficient in $W(\bullet)$. It turns out that the coefficient of each tree in $W(\bullet)$ has a combinatorial interpretation.

The first few iterated insertions of $\bullet$ are:
\begin{align*}
  \fN(\bullet) &= \bintree{[[][]]} \\
  \fN^2(\bullet)  &= 3 \bintree{[[[][]][]]} \\
  \fN^3(\bullet)  &= 12 \bintree{[[[[][]][]][]]} + 3 \bintree{[[[][]][[][]]]} \\
  \fN^4(\bullet)  &= 60 \bintree{[[[[[][]][]][]][]]} + 15 \bintree{[[[[][]][[][]]][]]} + 30 \bintree{[[[[][]][]][[][]]]}
\end{align*}
Hence, the terms in $W(\bullet)$ of degree up to $5$ are
\begin{align*}
W(\bullet) = \bullet
+ \frac{1}{2} \bintree{[[][]]} 
&+ \frac{1}{2} \bintree{[[[][]][]]} 
+ \frac{1}{2} \bintree{[[[[][]][]][]]} + \frac{1}{8} \bintree{[[[][]][[][]]]}  \\
&+ \frac{1}{2} \bintree{[[[[[][]][]][]][]]}  + \frac{1}{8}  \bintree{[[[[][]][[][]]][]]} + \frac{1}{4}  \bintree{[[[[][]][]][[][]]]} + \cdots.
\end{align*}
The coefficient of each term in $W(\bullet)$ is inverse to a power of two. This is formalized below.

\begin{proposition}
  The coefficient of $T$ in $W(\bullet)$ is $1/s_T$.
\end{proposition}

\begin{proof}
  If $T$ has $k$ leaves, then $T$ appears in $\fN^{k-1}(\bullet)$. Let $c_T$ be the coefficient of $T$ in $W(\bullet)$; then
  \[ \Big\langle \frac{1}{k!} \fN^{k-1}(\bullet), T \Big\rangle = c_T s_T \]
  and so
  \[ k! \cdot c_T s_T = \langle \fN^{k-1} (\bullet), T \rangle = \langle \bullet, \fP^{k-1}(T) \rangle, \]
  where the second equality is due to the adjointness of $\fN^{k-1}$ and $\fP^{k-1}$. Since $\fP^{k-1}(T) = k! \cdot \bullet$,
  \[ \langle \bullet, \fP^{k-1}(T) \rangle = k! \cdot s_\bullet = k!,\]
  so $k! \cdot c_T s_T = k!$, hence $c_T = 1 / s_T$.
\end{proof}

\section*{Acknowledgements}

I would like to thank my advisor Matilde Marcolli for her guidance, and Lo\"{i}c Foissy for providing helpful references to duality. I am also grateful to Yassine El Maazouz and Isabella Senturia for listening to my ideas at various stages of completion.

\bibliography{hopf}

@book {CartierPatras,
    AUTHOR = {Cartier, Pierre and Patras, Fr\'ed\'eric},
     TITLE = {Classical {H}opf algebras and their applications},
    SERIES = {Algebra and Applications},
    VOLUME = {29},
 PUBLISHER = {Springer, Cham},
      YEAR = {[2021] \copyright 2021},
     PAGES = {xv+268},
      ISBN = {978-3-030-77844-6; 978-3-030-77845-3},
   MRCLASS = {16T05 (16S30 16T15 16T30 16W10)},
  MRNUMBER = {4369962},
MRREVIEWER = {Stefaan\ Caenepeel},
       DOI = {10.1007/978-3-030-77845-3},
       URL = {https://doi.org/10.1007/978-3-030-77845-3},
}

@article{marcolli2015graph,
  title={Graph grammars, insertion {L}ie algebras, and quantum field theory},
  author={Marcolli, Matilde and Port, Alexander},
  journal={Mathematics in Computer Science},
  volume={9},
  number={4},
  pages={391--408},
  year={2015},
  publisher={Springer}
}

@article{oudom2008lie,
  title={On the {L}ie enveloping algebra of a pre-{L}ie algebra},
  author={Oudom, J-M and Guin, Daniel},
  journal={Journal of K-theory},
  volume={2},
  number={1},
  pages={147--167},
  year={2008},
  publisher={Cambridge University Press}
}

@article{hoffman2003combinatorics,
  title={Combinatorics of rooted trees and {H}opf algebras},
  author={Hoffman, Michael},
  journal={Transactions of the American Mathematical Society},
  volume={355},
  number={9},
  pages={3795--3811},
  year={2003}
}

@article{panaite2000relating,
  title={Relating the {C}onnes--{K}reimer and {G}rossman--{L}arson {H}opf algebras built on rooted trees},
  author={Panaite, Florin},
  journal={Letters in Mathematical Physics},
  volume={51},
  number={3},
  pages={211--219},
  year={2000},
  publisher={Springer}
}

@book{marcolli2025mathematical,
  title={Mathematical structure of syntactic merge: An algebraic model for generative {L}inguistics},
  author={Marcolli, Matilde and Chomsky, Noam and Berwick, Robert C},
  year={2025},
  publisher={MIT Press}
}

@phdthesis{bruned2015singular,
  title={Singular KPZ type equations},
  author={Bruned, Yvain},
  year={2015},
  school={Universit{\'e} Pierre et Marie Curie-Paris VI},
  note={https://theses.hal.science/tel-01306427v2/document}
}

@article{foissy2021algebraic,
  title={Algebraic structures on typed decorated rooted trees},
  author={Foissy, Lo{\"\i}c and others},
  journal={SIGMA. Symmetry, Integrability and Geometry: Methods and Applications},
  volume={17},
  pages={086},
  year={2021},
  publisher={SIGMA. Symmetry, Integrability and Geometry: Methods and Applications}
}

@article{foissy2025operads,
  title={Operads and bialgebras of multi-indices, and {N}ovikov algebras},
  author={Foissy, Lo{\"\i}c},
  journal={arXiv preprint arXiv:2510.18432},
  year={2025}
}

@article{schmitt1993hopf,
  title={Hopf algebras of combinatorial structures},
  author={Schmitt, William R},
  journal={Canadian Journal of Mathematics},
  volume={45},
  number={2},
  pages={412--428},
  year={1993},
  publisher={Cambridge University Press}
}

@article{foissy2013introduction,
  title={An introduction to {H}opf algebras of trees},
  author={Foissy, Lo{\"\i}c},
  journal={preprint},
  year={2013}
}

@incollection{connes1999hopf,
  title={Hopf algebras, renormalization and noncommutative geometry},
  author={Connes, Alain and Kreimer, Dirk},
  booktitle={Quantum field theory: perspective and prospective},
  pages={59--109},
  year={1999},
  publisher={Springer}
}

@book{yeats2017combinatorial,
  title={A combinatorial perspective on quantum field theory},
  author={Yeats, Karen},
  volume={15},
  year={2017},
  publisher={Springer}
}

@article{billera2001geometry,
  title={Geometry of the space of phylogenetic trees},
  author={Billera, Louis J and Holmes, Susan P and Vogtmann, Karen},
  journal={Advances in Applied Mathematics},
  volume={27},
  number={4},
  pages={733--767},
  year={2001},
  publisher={Elsevier}
}

@misc{OEIS,
  title        = {The {O}n-Line {E}ncyclopedia of {I}nteger {S}equences},
  author       = {OEIS Foundation Inc.},
  note         = {https://oeis.org}
}
\bibliographystyle{plain}

\end{document}